\documentclass[a4paper,12pt]{amsart}
\usepackage{amsmath,inputenc,euscript,amssymb}
\usepackage{color}
\title{The incompressible Navier-Stokes equations on non-compact manifolds}
\author{ V. Pierfelice }

\address{Universit\'e d'Orl\'eans \& CNRS,
F\'ed\'eration Denis Poisson (FR 2964) \& Laboratoire MAPMO (UMR 6628),
B\^atiment de math\'ematiques -- Route de Chartres,
B.P. 6759 -- 45067 Orl\'eans cedex 2 -- France}
 
\email{vittoria.pierfelice@univ-orleans.fr}

\date{\today}



\newtheorem{lemma}{Lemma}[section]
\newtheorem{theorem}[lemma]{Theorem}
\newtheorem{prop}[lemma]{Proposition}
\newtheorem{corollary}[lemma]{Corollary}

\newtheorem{remark}[lemma]{Remark}
\newtheorem{definition}[lemma]{Definition}

\newcommand{\grad}{\operatorname{grad}}
\newcommand{\Tr}{\operatorname{Tr}}
\newcommand{\Ric}{\operatorname{Ric}}
\newcommand{\dive}{\operatorname{div}}
\newcommand{\Riem}{\operatorname{Riem}}
\newcommand{\Image}{\operatorname{Image}}

\begin{document}

\begin{abstract}
We shall prove dispersive and smoothing estimates for Bochner type  laplacians on some non-compact Riemannian manifolds with negative  Ricci  curvature, in particular on hyperbolic spaces.

These estimates will be used to prove Fujita-Kato type theorems for the incompressible Navier-Stokes equations.

We shall also discuss the uniqueness of Leray weak solutions in the  two dimensional case.

\end{abstract}
\maketitle
\section{Introduction}

This work deals with  the equations describing  the motion of an incompressible fluid with viscosity  in a non-compact space $M$, more precisely,  
we shall  study the incompressible  Navier-Stokes  system on a non-compact Riemannian manifold.
Let us first  recall some classical results for the incompressible Navier-Stokes equations in the  flat case $\mathbb{R}^n$. 
 In this framework, 
the unknowns are  the velocity  $u: \mathbb{R}^{+}_{t}\times \mathbb{R}^n_{x} \rightarrow \mathbb{R}^n_{x}$ of the fluid,  a time dependent divergence free vector field on $\mathbb{R}^n$ and its  pressure $ p: \mathbb{R}^{+}_{t} \times \mathbb{R}^n_{x} \rightarrow \mathbb{R}$.
The  incompressible Navier-Stokes system  takes the following form 
\begin{equation}\label{system1}
 \left \{  \begin{array}{ll} \partial_{t} u +(u \cdot \nabla) u + \grad p =  \nu \Delta u, \\  \dive u = 0. \end{array} \right. 
\end{equation}
The velocity is divergence free because of the incompressibility assumption and    $\nu$ (which is the inverse of the Reynolds number when the system is written in non-dimensional coordinates) is positive since the fluid is viscous. Moreover, in cartesian coordinates the definitions of the  operators arising  in the previous system are: for all $ j \in \{1, \dots, n\} $ 
\begin{align*}
&  ((u\cdot \nabla) u)^{j} = (\nabla_{u} u)^j=  u^i \partial_{i}u^j, 
 &   div \, u = \partial_{i}u^i, \quad 
(\grad p)^j= \partial_{j}p, \quad 
& {(\Delta u)^j= \partial_{ii} u^{j}.} 
\end{align*}
where we sum over $i$.
We notice that, in cartesian coordinates, the vectorial laplacian is made by the usual (scalar) Laplacian acting on each component of vector fields $u = (u^1, \dots, u^n)$.

We  add to the system \eqref{system1} an initial condition on the velocity   $u_{|t=0}= u_{0},$ with  the initial data $u_{0}$ divergence free ($\dive \, u_{0}=0$).
The notion of $C^2$ solution (i.e. classical solution) is not efficient here. It has been pointed out by C. Ossen (see \cite{Oseen28} and \cite{Oseen29}) that another concept of solution must be used.
There are many notions of solutions that are appropriate for this system. The most famous ones are the Leray weak solutions that are based on the  energy dissipation (see \cite{Leray26}) and the  Kato type solutions that are based on the  scaling of the equation (see \cite{kato}).
 One way of studying the initial value problem (NSE) is via the weak solutions introduced
by Leray. Indeed, Leray and Hopf showed the existence of a
global weak solution of the Navier-Stokes equations corresponding to initial data
in $L^2(\mathbb{R}^n)$ (see \cite{Leray26}, \cite{hopf}).  Lemari\'e extended this construction and obtained the existence of uniformly
locally square integrable weak solutions. Questions about the uniqueness
and regularity of these solutions are completely clear only when $n=2$. In particular, in dimension $2$ the energy inequality is verified and Leray weak solutions are unique for any initial data $u_0 \in L^{2}(\mathbb{R}^2)$ and global  propagation of higher regularity holds.
 When $n\geq 3$, these questions have not been answered yet; the case of dimension 3  is one of millenium problems. But important
contributions in understanding partial regularity and conditional uniqueness
of weak solutions should be mentioned  (see e.g.  \cite{CKN}, \cite{Lin}, \cite{LM}, \cite{FurLemTer},  \cite{EusSESve}). 
 Because of the uniqueness problem with the weak solution in dimension $n \geq3,$ another approach was introduced by Kato and Fujita (1961) studying  stronger solutions (or mild solutions) (see \cite{FujitaKato}).
 To define them, they use the Hodge decomposition in $\mathbb{R}^n$, i.e. for every $L^2$ vector field, one has the unique orthogonal decomposition
 \begin{equation}
 \label{dechodge} u= v + \grad  q, \quad \dive  v= 0, \quad v= \mathbb{P} u,
 \end{equation}
where $\mathbb{P}$ is the Leray projector on divergence free vector  fields.
 Formally if  $u$ solves the Navier-Stokes Cauchy problem, then applying the projector $\mathbb{P}$ to the equation, the pressure term dissapear and one gets the following Cauchy problem for this semi-linear parabolic system
\begin{equation}
\label{katoform}
 \left \{ \begin{array}{ll}  \partial_{t} u  - \Delta u= - \mathbb{P}(u \cdot \nabla u), \\
  u_{|t=0}= u_{0}. \end{array} \right.
\end{equation}
 By using the heat semi-group and the Duhamel formula, the PDE is reformulated as a fixed point problem in a suitable Banach space $X_T$
$$u(t, \cdot)= e^{t \Delta } u_{0} - \int_{0}^t e^{(t-s) \Delta} \mathbb{P}(u \cdot \nabla u) \, ds.$$
Then the strong solutions are solutions of this fixed point problem. This approach of Kato allows us to get the well-posedness of the Cauchy problem to the Navier-Stokes equations locally in time and globally for small initial data in various subcritical or critical spaces.
 The critical  spaces are the natural ones to solve the equation by the fixed point method since they are  invariant under the scaling of the equation: if $u$ is the solution, then $u^\lambda = \lambda u(\lambda^2 t, \lambda x)$ is also a solution. 
 The main results  for the critical spaces are the following: $\dot{H}^{- 1 + {n \over 2}} \subset L^n \subset \dot B^{-1 + {n \over p}}_{p, \infty} \subset BMO^{-1},$ obtained by   \cite{FujitaKato},  \cite{Giga-Miyakawa},  \cite{kato}, \cite{Weissler}, \cite{CannoneMeyer},  \cite{Planchon},  \cite{KochTataru}, \cite{Auscher}.
  The largest critical spaces is $\dot B^{-1}_{\infty, \infty}$, but the Cauchy problem is showed to be ill-posed by \cite{Bourgain-Pavlovic}. 
  Moreover, there are  global well-posedness results for some classes of large data in all the above spaces, that uses the structure of the non-linear term (see for examples \cite{Ladyzhenskaya}, \cite{Abidi}, \cite{CheminGallagher}).

  In this paper, we shall mainly  be interested in the Kato approach in the case of the Navier-Stokes equations on non-compact riemannian manifolds.
   The plan of the paper is as follows. In the next section,  we give a more precise description of the manifolds that we shall consider and we  recall some definitions and properties  of Riemannian geometry  and 
    functional analysis on these  non-compact manifolds. In section \ref{sectionnavierstokes}, we recall the natural
     way to write the Navier-Stokes equation on a non-flat manifold that was pointed out by \cite{EbinMarsden}, \cite{Taylor}. Note that the issue  is
      that we need a Laplacian acting on vector fields  and that there is no canonical object of this type on a manifold
       (there are many possibilities such as the Hodge Laplacian, the Bochner Laplacian). We shall also explain a  good
        way to write the system under the form \eqref{katoform} on our manifolds. Note that we cannot use directly the decomposition
         \eqref{dechodge}  that does not hold in general on a manifold when non-trivial $L^2$ harmonic 1-forms are present.
          This phenomenon is at the origin of the non-uniqueness phenomenon on the hyperbolic plane pointed out in \cite{Czubak}, \cite{Khesin} and produce non-unique $\mathcal{C}^\infty$
          solutions. In section \ref{sectiondispersive}, we prove dispersive  and smoothing  estimates
   for the heat  and Stokes equations associated to the Bochner Laplacian. The negative curvature yields better large time decay than in the Euclidian case.
    Our set of estimates for the Stokes problem  is more complete when the Ricci curvature of the manifold is constant (thus in particular
     on the hyperbolic spaces and also on more general symmetric spaces of non-compact type). This comes from the fact that in this case the study of the Stokes problem can
      be reduced to the study of the vectorial heat equation.
    These are the crucial estimates
    needed in order to get Fujita Kato type theorems. In section \ref{sectionKato}, we prove well-posedness
     results for the Navier-Stokes equations in an $L^n$ framework. Finally in section \ref{sectionLeray}, 
     we discuss how by eliminating the pressure from the Navier-Stokes system our approach  can be
      used to recover the uniqueness of Leray weak solutions on two-dimensional non-compact manifolds.  This gives
       another approach to the recent result \cite{CzubakChan}.

\section{Baby Geometry $ \clubsuit$ }

We shall recall in this section the main objects of Riemannian geometry and their properties that we need. For more details,  we refer to  Riemannian geometry  textbooks \cite{GHL},\cite{Jost} for example.
\subsection{Connections}
We consider $(M,g)$ a Riemannian manifold. We shall denote by $\nabla$ the Levi-Civita connection:
\begin{align*}
 \Gamma(TM) \times \Gamma(TM) &  \rightarrow \Gamma(TM)  \\
  (X,Y) \mapsto \nabla_X Y
  \end{align*}
  where we denote by  $\Gamma(TM)$  the set of vector fields on $M$.
   The crucial property of this connection is its compatibility with the metric: for any vector fields $X, \, Y, \, Z$, we have
   \begin{equation}
   \label{levi}
     X\cdot g(Y,Z)= g( \nabla_{X} Y, Z) +  g( Y, \nabla_{X} Z).
     \end{equation}

  For $X\in  \Gamma(TM)$, we can extend $\nabla_X$ to arbitrary $(p,q)$ tensors   by requiring that
  \begin{enumerate}
\item[i)] $\nabla_X (c(S)) = c(\nabla_XS)$
 for any contraction $c$,
 \item[ii)]$ \nabla_X( S\otimes T) = \nabla_X S \otimes T +  S\otimes \nabla_X T$
\end{enumerate}  
with the convention that for $f$ a function $\nabla_X f= X \cdot f.$
In particular, we get that for  $S\in \Gamma(\bigotimes^p(TM) \bigotimes^q(T^* M))$
$$( \nabla_X S)(X_1, \cdots X_q)= \nabla_X \big(  S(X_1, \cdots X_q) \big) -S(\nabla_X X_1, \cdots,  X_q)- \cdots -
 S(X_1, \cdots , \nabla_X X_q).$$

We define the covariant derivatives $\nabla$ on tensor field  $S\in \Gamma(\bigotimes^p(TM) \bigotimes^q(T^* M))$ by
$$ \nabla S (X, X_1, \cdots X_q) =( \nabla_X S)(X_1, \cdots X_q),  $$
thus $\nabla S \in  \Gamma(\bigotimes^p(TM) \bigotimes^{q+1}T^* M)) $.

\subsection{Curvatures}
 We shall use the following classical definitions for the various curvature tensors.
 The curvature tensor is defined by 
 \begin{equation}\label{tensorecurv}
  R(X, Y)Z= - \nabla_X  (\nabla_Y Z) + \nabla_Y
( \nabla_X Z) + \nabla_{[X,Y]} Z, \quad \forall X, \, Y, \, Z \in \Gamma(TM).
\end{equation}
The Riemann curvature tensor is given by 
\begin{equation}\label{tensRiem}
 \Riem(X,Y,Z,T)= g(R(X,Y) Z, T), \quad \forall X, \, Y, \, Z, \, T \in \Gamma(TM)\end{equation}
 and the Ricci curvature tensor is defined by 
 \begin{equation}\label{tensRic} 
 \mbox{Ric} (X,Y)= \sum_{i=1}^n  \Riem(X,e_i,Y,e_i), \end{equation}
  for an orthonormal basis $(e_1, \cdots e_n)$.
   The notion of sectional curvature will be also used. For every $(X,Y) \in (T_{x} M)^2$, we define the sectional curvature of  the plane $(X,Y)$
    as
   $$ \kappa (X,Y)=  {R(X,Y, X, Y) \over  g(X,X) g(Y,Y) - g(X,Y)^2}.$$
   
  \subsection{Metric on tensors}
  Let us recall the musical applications: for a 1-form $\omega$, we associate the vector field $\omega^\sharp$ defined by
  $$ g( \omega^\sharp, Y) = \omega(Y), \quad \forall Y \in \Gamma(TM)$$
  and for a vector field $X$, we associate the 1-form $ X^\flat$ defined by 
  $$ X^\flat (Y)= g(X, Y),\quad \forall Y \in \Gamma(TM).$$
  
  The Riemmanian gradient of a function is then defined as 
  $$ \grad p =( d p )^\sharp.$$ 
   
  More generally, for tensors $T \in  \Gamma(\otimes^p TM \otimes^q T^* M), $ we  have
\begin{align*}
& T^\sharp= C_{1}^2( g^{-1} \otimes T) \in  \Gamma(\otimes^{p+1} TM \otimes^{q-1} T^* M),   \\
&  T^\flat= C_{2}^1( g \otimes T) \in  \Gamma(\otimes^{p-1} TM \otimes^{q+ 1} T^* M), \\
& \dive T= C^1_1 \nabla T \in  \Gamma(\otimes^{p-1} TM \otimes^{q} T^* M)
\end{align*}
where $C^i_j$ stands for the contraction of the $i$ and $j$ indices for tensors.

We can define a metric on 1-forms by setting
$$ g(\omega, \eta):= g( \omega^\sharp, \eta^\sharp), \quad  \forall \omega, \, \eta \in \Gamma(T^*M).$$
We can then extend the definition to general tensors fields in $  \Gamma(\otimes^p TM \otimes^q T^* M), $
 by setting
 $$ g:= (\otimes^p g) \otimes (\otimes^q g).$$
 In local coordinates $(x^1, \dots, x^n)$, for  $T, S \in  \Gamma(\otimes^p TM \otimes^q T^* M), $ i.e. 
 $$T = T_{j_1 \dots j_q}^{i_1 \dots i_p} \,\, \partial_{x^{i_1}} \otimes  \dots  \otimes \partial_{x^{i_p}}\otimes dx^{j_1} \otimes \dots \otimes dx^{j_q},$$  
  $$S = S_{j'_1 \dots j'_q}^{i'_1 \dots i'_p} \,\, \partial_{x^{i'_1}} \otimes  \dots  \otimes \partial_{x^{i'_p}}\otimes dx^{j'_1} \otimes \dots \otimes dx^{j'_q},$$ 
   this yields the expression  $$g( T, S) = g_{i_1 i'_1} \cdots g_{i_p i'_p} g^{j_1 j'_1} \cdots g^{j_q j'_q} T_{j_1 \dots j_q}^{i_1 \dots i_p} S_{j'_1 \dots j'_q}^{i'_1 \dots i'_p}.$$
   We shall also use for tensors the notation 
 \begin{equation}\label{norm2}
 | T |^2 =g(T,T).
  \end{equation}
  We define the Sobolev norms of tensors
   {$T \in \Gamma( \otimes^p TM \otimes^q T^* M)$} by
   {$$ \| T\|_{W^{m, p}}= \Big( \sum_{0 \leq  k \leq m } \int_{M} g(\nabla^k T, \nabla^k T)^{p \over 2}\, dvol \Big)^{1\over p}, \quad
  W^{m,2}= H^m.$$}

  \subsection{Normal coordinates} To compute intrinsic objects in local coordinates, it will be very often convenient to use
   normal coordinates. More precisely, we shall use that in the vicinity of any point $m_0$, there exists a coordinate system
     $(x^1, \cdots, x^n)$ such that at the point $m_0$ the coordinates of the Riemannian metric and the Christoffel coefficients verify
     \begin{equation}
     \label{normal}
       g_{ij}(m_0) = \delta_{ij}, \quad \quad    \Gamma_{ij}^k(m_0)= 0.
     \end{equation}
     
     \subsection{Some useful geometric formulas}
     \begin{lemma}[Kato inequality]
     For any vector field $u$
     \begin{equation}\label{katoin}
      | \nabla |u| | \leq  | \nabla u|.
      \end{equation}
     
     \end{lemma}
     \begin{proof}
      We prove the inequality at each point $m$ by using a normal coordinate system centered at $m$.  Let us set  $e_{i}= \partial/\partial x^i= \partial_{i}$, then  $(e_{1}, \cdots, e_{n})$ is an orthonormal basis at $m$. By using \eqref{normal} and by Cauchy-Schwartz inequality we have on the one hand
$$
      |  \nabla (|u|^2) |_{/m} ^2=   \sum_i( \partial_i |u|^2)^2_{/m}= 4 \sum_i
               g( \nabla_{e_i} u , u)^2_{/m} \leq 4 | \nabla u|^2_{/m}   |  u|^2_{/m},
         $$
         on the other hand
        $$\nabla (|u|^2) |_{/m} ^2= 4  |( \nabla (|u|) |u|) |^2_{/m},$$
     so we obtain 
     $$  |( \nabla (|u|) |u|) |^2_{/m}\leq  | \nabla u|^2_{/m}   | u|^2_{/m}$$
     which yields the result.
          \end{proof}
Let us denote by $\Delta_{g}$ the Laplace Beltrami  operator and 
  by     $\overrightarrow{\Delta}$  the Bochner Laplacian,
\begin{equation}\label{bochn}
 \overrightarrow{\Delta} u  = - \nabla^* \nabla u = \Tr_{g}(\nabla^2 u)
 \end{equation}
 where 
 $\nabla^*$ is the formal adjoint of $\nabla$ for the $L^2$ scalar product and
$$  \Tr_{g}(\nabla^2 u)=g^{ij}\nabla^2 u(e_{i}, e_{j}) $$
in local coordinates.

    \begin{lemma}\label{bochner}
      For any vector field $u$, we have  Bochner's identity
    \begin{equation}\label{Bochner}
    {1 \over 2 } \Delta_{g} \big( g(u,u))=  g(\overrightarrow \Delta u, u) +  g( \nabla u, \nabla u).
    \end{equation}
      
     \end{lemma}
     Note that in the right hand side the scalar product $ g( \nabla u, \nabla u)$ is  the scalar product on $(1,1)$ tensors
     defined above.
     
     \begin{proof}
     To prove the formula, we shall  compute each term in the formula in normal coordinates at $m$ for any point $m$.
     Let us set  $e_{i}= \partial/\partial x^i= \partial_{i}$, then $(e_{1}, \cdots, e_{n})$ is an orthonormal basis at $m$.
          By using the properties  \eqref{normal} of the normal coordinates at $m$, we have
     $$(\overrightarrow \Delta u) _{/m} = ( \nabla^2 _{e_{i}, e_{i}} u)_{/m}=  \big(\nabla_{e_{i}} \big( \nabla_{e_{i}} u ) \big)_{/m}$$
      Therefore,  we obtain by using \eqref{levi}
      $$ g( \overrightarrow \Delta u, u)_{/m} =  \Big(  \partial_{i} \big( g(\nabla_{e_{i}} u, u ) \big)  - g( \nabla_{e_{i}}u, \nabla_{e_{i}} u)
       \Big)_{/m}$$
        and hence
        $$  \Big(g( \overrightarrow \Delta u, u) +  g( \nabla u, \nabla u) \Big)_{/m}=    \big(\partial_{i} \big( g(\nabla_{e_{i}} u, u ) \big) \big)_{/m}$$
        by using again \eqref{normal}. To conclude, 
        we observe  by using again \eqref{levi} that
        $$  g(\nabla_{e_{i}} u, u ) = {1 \over 2} \partial_{i} \big( g(u, u))$$ and then that
        $$    \big(\partial_{i} \big( g(\nabla_{e_{i}} u, u ) \big)_{/m}  =  {1 \over 2 } \big(\partial_{i}^2 ( g(u, u) \big)_{/m} = {1 \over 2}
         \Delta_{g} \big( g(u, u) \big)_{/m}.$$
        
     \end{proof}

      \subsection{Functional Analysis on non-compact manifolds}
     
   In all this paper,  we shall consider  smooth, complete, non-compact, simply connected  Riemannian manifolds $M$ of dimension $n\geq 2$ that verify the following assumptions
\begin{itemize}
 \item {\bf (H1)} $ |R | +  |\nabla R|  + |\nabla^2 R |\leq K$;
 \item {\bf (H2)}   $-{1 \over {c_0}} g \leq \Ric \leq - c_0 g$, for some $c_0>0$;
 \item {\bf(H3)} $\kappa \leq 0$;
  \item {\bf(H4)} $ \inf_{x \in M} r_x >0;$
\end{itemize}
where $R$ is the curvature tensor, $\Ric$ is the Ricci curvature tensor, $\kappa$ is the sectional curvature and $r_x$ stands for the injectivity radius for the exponential map at $x$.
\begin{remark}\label{REMP}
This set of assumptions have several important consequences, which will be crucial in the following.
\begin{enumerate}
\item{$C_c^\infty(M)$ is dense in $H^1(M)$, (see \cite{Hebey});}
\item{From Varopoulos \cite{Varopoulos}, (see also \cite{Hebey})  the Sobolev inequalities are verified. In particular 
 $$ \eta_n \| f\|^2_{L^{2^*}(M)}  \leq  \| \nabla f \|^2_{L^2(M)}, \quad \forall f \in H^1(M)$$
 is verified for some $\eta_n>0,$  where $2^*= 2n/(n-2)$ if $n\geq 3$ and $2^*$ is arbitrary in  $ (2, + \infty)$ when $n=2$.}
 \item{ In dimension $n=2,$ we also have  the continuous embedding $W^{1, 1}(M) \subset L^2(M)$, therefore there exists $C>0$ such that $$\| f \|^2_{L^2(M)} \leq C \left( \| \nabla f \|_{L^1(M)} + \| f \|_{L^1(M)}\right).$$
 By using this inequality with $f = |g|^2$ and by the Cauchy-Schwarz inequality we obtain the following Gagliardo-Nirenberg inequality  $$\| g\|^4_{L^4(M)}  \leq C \left( \| \nabla g \|^2_{L^2(M)}  \| g \|^2_{L^2(M)} +  \| g \|^4_{L^2(M)}\right).$$}
 \item{From Setti \cite{Setti}, (see also \cite{McKean}) the Poincar\'e inequality 
\begin{equation*}
\delta_n \| f \|^2_{L^2(M)}  \leq  \| \nabla f \|^2_{L^2(M)}, \quad \forall f \in H^1(M)
\end{equation*}  is verified for $\delta_n \geq [c_0 - (n-1)(n-2) \kappa^*]/4 >0$, with $\kappa^*= \sup_M \kappa$.}
\end{enumerate}
\end{remark}

\begin{remark}\label{remHyp}
An important example of non-compact Riemannian manifolds for which our hypothesis {\bf{(H1-4)}} hold true  are  the well-known  real hyperbolic spaces $M = \mathbb{H}^n(\mathbb{R}), \, n\geq 2,$ defined as follows
$$ \mathbb{H}^n= \left\{ \Omega= (\tau, x)\in \mathbb{R}^{n+1}, \,  \Omega = (\mbox{cosh }r, \, \omega \, \mbox{sinh }r), \, r \geq 0, \, \omega \in \mathbb{S}^{n-1} \right\},$$
the metric  $g$ being
$$ g=  dr^2 +  (\mbox{sinh }r )^2 d\omega^2 $$
with $d\omega^2$ the canonical metric on the sphere $\mathbb{S}^{n-1}$.

The Ricci curvature tensor is constant, $\Ric = \kappa (n-1) g$ with $\kappa$  the sectional curvature given by   $\kappa=-1$.
In fact, the  curvature tensor is  $$R(X,Y) Z = \kappa R_{0}(X,Y) Z, \quad \forall X,Y,Z \in \Gamma(T \mathbb{H}^n)$$ where $R_{0}(X,Y) Z: = g(X,Z) Y - g(Y,Z) X$ (which  also implies that $\nabla R =0$)  and thus  
  the Riemann tensor is
 $$ \Riem( X, Y, Z, T) = k g(R_{0}(X,Y) Z, T) = - \left[ g(X,Z) g(Y,T) - g(Y,Z) g(X,T) \right].$$
 
  Assumptions   {\bf{(H1-4)}} are also verified by several  other classical examples in geometry, like 
  some  Damek-Ricci spaces and all symmetric spaces of non-compact type (see \cite{DR}, \cite{Uno}, \cite{H}, \cite{Erb}).
 \end{remark}

     \section{The Navier-Stokes equations on manifolds}\label{sectionnavierstokes}
     The Navier-Stokes equations on a Riemannian manifold $(M,g)$ takes the form 
     $$ \partial_t u + \nabla_u u + \grad p= \nu Lu, \quad \dive u= 0,$$
     where the diffusive part is defined  by the  operator $L$. The unknowns $(u, p)$ are such that the velocity $u(t, \cdot) \in \Gamma(TM)$ is a vector field on M and the pressure $p(t, \cdot)$ is a real-valued function. For the left hand side of the equation all terms have a natural definition. Indeed, $\nabla_u u \in  \Gamma(TM)$ stands for the covariant derivative of $u$ along of $u$ and  $\grad p$ is the Riemannian gradient of the pressure. Note that  since $u$ is divergence free, we have also the  following identity 
     \begin{equation}\label{NLdiv}
      \nabla_u u= \dive (u\otimes u).
     \end{equation}
     
      To define the vectorial Laplacian $L$, we have to make a choice since there
      is no canonical definition of a Laplacian on vector fields on Riemannian manifolds: 
       there are at least two candidates for the role of Laplace operator, i.e. the Bochner and Hodge Laplacians.
Following \cite{EbinMarsden}, \cite{Taylor} (see also \cite{Priebe}, \cite{MitreaTaylor}),  the correct formulation is obtained by  introducing the stress tensor. Let us recall that 
on $ \mathbb{R}^n$,  if  $ \dive u=0$, we have   
$$ Lu  = \dive \left(  \nabla u + \nabla u^t \right) =  \Delta u.$$
 The natural generalization on M 
  is to take
  $$   Lu =  \dive \left(  { \nabla u }+ \nabla u^t \right)^\sharp  \in \Gamma (TM). $$
Since  $u$ is divergence free, we can  express $ {L}$ in the following way:
$$  Lu= \overrightarrow{\Delta} u + r(u), $$
where
  $r$ is the Ricci operator which is related to the  Ricci curvature tensor by 
$$  r(u) =  \big(\Ric (u, \cdot) \big)^\sharp \in \Gamma (TM).$$ 
By using the Weitzenbock formula on $1-$forms 
  \begin{equation}\label{WEZ}
   \Delta_{H} u^\flat= \nabla^* \nabla u^\flat + \Ric (u, \cdot),
  \end{equation}
where  $\Delta_{H} = d^* d + d d^*$ is the Hodge Laplacian on 1-forms, we can also relate $L$ to the Hodge Laplacian:   
$$ L u= \Big(-\Delta_{H} u^\flat + 2 \Ric (u, \cdot) \Big)^\sharp.$$
Let us consider the Cauchy problem for the incompressible Navier-Stokes equation on $M$ (assume $\nu=1$)
 \begin{equation}\label{NS}
 \left\{ \begin{array}{ll}   \partial_t u + \nabla_u u + \grad p=  \overrightarrow{\Delta} u + r(u),\\  \dive u= 0, \\   u_{|t=0}= u_{0},  \quad u_0 \in \Gamma(TM).
 \end{array} \right.
  \end{equation}
In view of its own structure and by {\bf (H2)}, we can deduce that the smooth solution of \eqref{NS} satisfies the following energy inequality
\begin{equation}\label{EN}
 \|u(t)\|_{L^2}^2 + \int_{0}^t\big( \| \nabla u (s) \|_{L^2}^2  +  c_0 \,  \|u(s) \|_{L^2}^2 \big)\, ds \leq  \|u_{0}\|_{L^2}^2.
\end{equation}
Indeed, multiplying $u$ in \eqref{NS} and then integrating on $M$  by part combining with the Bochner identity \eqref{bochner}, we have \eqref{EN}. 
According to \eqref{EN}, it is natural to construct  weak solution that verify the energy inequality.  Nevertheless we expect at least the same
 difficulty as in the Euclidean case (at least in dimension greater than  $3$) and hence it is also  natural to study Kato type solutions.
 In both cases, one has to be careful when eliminating the pressure. Indeed,  
in the Euclidean case for smooth solutions it is
well known that the pressure term $p$ can be eliminated via Leray-Hopf projections
and that we can view Navier-Stokes system \eqref{system1} as an evolution equation of $u$ alone.
On a Riemannian manifold $M$ some problems  may occur since the Kodaira-Hodge decomposition of  $L^2$  1-forms on complete manifolds
 is under the form
$$ L^2\big(\Gamma (T^*M) \big)= \overline{\Image d} \oplus \overline{\Image d^* } { \oplus \mathcal{H}^1(M)}$$
where $\mathcal{H}^1(M)$ is the space of $L^2$ harmonic 1-forms (see \cite{Kodaira}). It may happen that there are non-trivial $L^2$ harmonic 1-forms which
are responsable for non-uniqueness (even in dimension two,  \cite{Czubak}, \cite{Khesin} on the hyperbolic space $\mathbb{H}^2$). We shall make the following choice for the pressure in order to eliminate this non-uniqueness phenomenon.
 We first note that if $(u,p)$ is a smooth solution of the Navier-Stokes equation \eqref{NS}, then
by taking the divergence of the first equation in \eqref{NS} and  by noticing that $ \dive u= 0$, we obtain that
\begin{equation}\label{press}
\Delta_g p + \dive  \left[ \nabla_u u \right] - 2 \dive (ru) = 0.
 \end{equation}
We used  the consequence of Weitzenbock formula \eqref{WEZ} that $\dive ( \overrightarrow{\Delta} u) = \dive (ru)$ if $ \dive u= 0$.
In order to determine the pressure, we shall always choose the solution in $L^p$ of this elliptic equation since 
 $ \Delta_g :\, W^{2,p} \rightarrow L^p$ is an isomorphism thanks to the assumptions {\bf(H1-H4)} (see \cite{Lohoue}, \cite{Strichartz} for $2 \leq p <\infty$).
It follows that 
\begin{equation*}
 \grad p = \grad  (-\Delta_g)^{-1}\dive \left[ \nabla_u u \right] - 2 \grad (-\Delta_g)^{-1}\dive (ru).
 \end{equation*}
 We shall discuss why this choice is appropriate to get uniqueness results (in relation with the counterexamples of  \cite{Czubak}, \cite{Khesin})
  in section \ref{sectionLeray}.
It will be convenient to use the notation  $$\mathbb{P} = I + \grad (-\Delta_g)^{-1}\dive .$$
Solving $p$ from \eqref{press} and inserting it into \eqref{NS}, we get
 \begin{equation}\label{NSprew}
 \left\{ \begin{array}{ll}   \partial_t u  -  \overrightarrow{\Delta} u - r(u)  + 2 \grad (-\Delta_g)^{-1}\dive (ru) =   - \mathbb{P}\left[ \nabla_u u \right] ,\\  \dive {u_0}= 0, \\   u_{|t=0}= u_{0}, \quad u_0 \in \Gamma(TM)
 \end{array} \right.
  \end{equation}
    From \eqref{NSprew}, we see that the Navier-Stokes system on $M$ belongs to a class of non-linear parabolic equations on vector fields.
    
    We remark that when the Ricci tensor $Ric$ is a negative constant scalar multiple of the metric and  $\dive u=0$, we have $\dive (ru)=0 $ and the linear non-local term disappear.
    In particular this occurs on the hyperbolic spaces $M = \mathbb{H}^n(\mathbb{R})$ (see Remark \ref{remHyp}).
        
  In order to use the fixed point method, we need to prove dispersive and smoothing estimates for the  semi-group associated to the linear part of the  Cauchy problem \eqref{NSprew}.

\section{Dispersive and smoothing estimates}\label{sectiondispersive}
\subsection{The case of vectorial heat equations}
We study the Cauchy problem for the heat equation associated to the Bochner Laplacian on vector fields:
\begin{equation}\label{Bochnerequation}
 \left\{ \begin{array}{ll} \partial_{t} u = \overrightarrow{\Delta} u + r (u), \\
   u_{|t=0}= u_{0}, \quad u_0 \in \Gamma(TM).
   \end{array} \right.
\end{equation}

We shall prove dispersive and smoothing estimates for the semi-group  associated to this  vectorial heat equation
\eqref{Bochnerequation}
on  $M$ of dimension $n\geq 2$ satisfying our assumptions {\bf (H1-4).}

These kind of estimates are   related   to the behaviour of the heat kernel which is well studied 
 in the literature for various types of manifolds  for  both the Laplace-Beltrami  and the Hodge Laplacian (see for example
\cite{Li-Yau}, \cite{AnkerJi}, \cite{APV1}, \cite{Grigoryan-Noguchi}, \cite{Lohoue},  \cite{Ouhabaz},   \cite{Varopoulos},  \cite{Carron}, \cite{Coulhon},  \cite{Auscher-Coulhon-Duong-Hoffman}, \cite{Pedon}, \cite{Bakry} and others)

The main results of this section are:
\begin{theorem}[Dispersive estimates]
\label{theolin1}
Assuming {\bf (H1-4)},  the solution of \eqref{Bochnerequation} satisfies the following dispersive estimates 
\begin{equation}\label{dispersiveestimates}
\| e^{t ( \overrightarrow{\Delta} + r)} u_{0} \|_{L^q} \leq c_n(t)^{ \left({1\over p} - { 1 \over q } \right )} \, e^{-  t  ( \gamma_{p,q} + c_0) } \, \| u_{0} \|_{L^p},
\end{equation}
for every   $p,q$ such that $ 1 \leq p \leq q \leq  +\infty, $
 with \, $ \gamma_{p,q}=  \frac{\delta_n}{2} \left[ \left( {1 \over p } - {1 \over q} \right) + { 8  \over q} \left( 1 - { 1 \over p } \right) \right]$, $ c_n(t)=  C_{n} \max \big({1 \over  t^{n\over 2}}, 1 \big)$ and for all $u_{0} \in L^p(\Gamma(TM)).$

\end{theorem}

\begin{theorem}[Smoothing estimates]
\label{theolin2}
Assuming  {\bf (H1-4)}, the solution of \eqref{Bochnerequation} satisfies the following smoothing estimates
\begin{equation}\label{smoothingestimates}
 \| \nabla u (t) \|_{L^p } \leq { C  \max{\left(\frac{1}{\sqrt{t}}, 1\right)} e^{-t \,\left(c_0 +\frac{4 \delta_n}{p}\left(1-\frac{1}{p}\right)\right) }} \| u_0 \|_{L^p}
 \end{equation}
for every $ 1 < p  < + \infty$ and for all $u_{0} \in L^p(\Gamma(TM)).$
\end{theorem}

 Under the same assumptions as in Theorem \ref{theolin2}, we can deduce more general smoothing estimates $L^p \to W^{1,q}$.
 
  \begin{corollary} \label{corlin}
 Assuming {\bf (H1-4)}, for every   $p,q$ 
 such that $ 1 < p \leq q <  +\infty $,  we obtain for all times $t>0$
 \begin{equation}
  \|  \nabla e^{t (\overrightarrow{\Delta} +r) } u_{0} \|_{L^q} \leq  \, c_n(t)^{ \left({1\over p} - { 1 \over q } + {1 \over n} \right )} \, 
  e^{- t  \, \left(c_0 +  {\gamma_{q,q} + \gamma_{p,q} \over 2} \right) }
\|u_{0}\|_{L^p}, 
 \end{equation}
 with \, $ \gamma_{p,q}=  \frac{\delta_n}{2} \left[ \left( {1 \over p } - {1 \over q} \right) + { 8  \over q} \left( 1 - { 1 \over p } \right) \right]$, $ c_n(t)=  C_{n} \max \big({1 \over  t^{n\over 2}}, 1 \big)$ and for all $u_{0} \in L^p(\Gamma(TM)).$ 
 Moreover, under the same assumption, we have 
 \begin{equation}\label{smot}
\|   e^{t ( \overrightarrow{\Delta} +r) } \nabla^* T_{0} \|_{L^q} \leq  \, c_n(t)^{ \left({1\over p} - { 1 \over q } + {1 \over n} \right )} \, 
  e^{- t  \, \left(c_0 +  {\gamma_{q,q} + \gamma_{p,q} \over 2} \right) }
\|T_{0}\|_{L^p}, 
 \end{equation}
 with \, $ \gamma_{p,q}=  \frac{\delta_n}{2} \left[ \left( {1 \over p } - {1 \over q} \right) + { 8  \over q} \left( 1 - { 1 \over p } \right) \right]$, $ c_n(t)=  C_{n} \max \big({1 \over  t^{n\over 2}}, 1 \big)$ and for all tensors $  T_{0} \in L^p(\Gamma(TM \otimes T^*M))$

\end{corollary}
\begin{proof}
It is sufficient to use the semi-group property combined with the smoothing \eqref{smoothingestimates} and dispersive estimates \eqref{dispersiveestimates}
$$ e^{t (\overrightarrow{\Delta}+r) } = \, e^{{t\over2}(\overrightarrow{\Delta}+r) } \, (e^{{t\over2} (\overrightarrow{\Delta} +r)} ): L^p \rightarrow L^q \rightarrow W^{1,q}.$$

The second estimate follows by duality. Note that $r$ is symmetric for the metric $g$ by definition.
\end{proof}

\begin{remark}
 Note that, since
 $$ \nabla^* T =- \dive (T^\sharp), \quad \forall T \in  \Gamma(TM \otimes T^*M) $$
 from \eqref{smot} we also get  the following smoothing estimate
  \begin{equation} \label{smononl}
  \| e^{t (\overrightarrow{\Delta}+r) } \dive T^\sharp_{0} \|_{L^q} \leq  \, c_n(t)^{ \left({1\over p} - { 1 \over q } + {1 \over n} \right )} \, 
  e^{- t  \, \left( c_0+ {\gamma_{q,q} + \gamma_{p,q} \over 2} \right) }
\|T^\sharp_{0}\|_{L^p}, 
\end{equation}
 with \, $ \gamma_{p,q}=  \frac{\delta_n}{2} \left[ \left( {1 \over p } - {1 \over q} \right) + { 8  \over q} \left( 1 - { 1 \over p } \right) \right]$, $ c_n(t)=  C_{n} \max \big({1 \over  t^{n\over 2}}, 1 \big)$ and for all tensors $T^\sharp_{0} \in L^p(\Gamma (TM \otimes TM)).$
\end{remark}

\subsection{Proof of Theorem \ref{theolin1}}
We shall split the proof of Theorem \ref{theolin1} in several steps.
 We shall first start with a comparison lemma that allows to reduce the proof of estimates for the vectorial Laplacian heat equation
 to estimates for the Laplace Beltrami heat equation.
 \begin{lemma}
 \label{comparaison}
  For any $u_0 \in  C^\infty_b(\Gamma(TM))$, we have the pointwise estimate
 $$ | e^{t ( \overrightarrow{\Delta} + r) } u_{0} |_{/x} \leq e^{t  (\Delta_{g}  - c_0) } |u_{0}|_{/x}, \quad \forall x \in M$$ 
 \end{lemma}

\begin{proof}
Let  $u(t,x) =\big( e^{t  (\overrightarrow{\Delta}  + r)} u_{0} \big)(x)$ be the solution of  the vectorial heat equation \eqref{Bochnerequation}. We note that   $ |u|= g(u,u)^{1\over2}$ solves the following scalar heat equation
$$ \partial_{t} |u| - \Delta_{g} |u|= {1 \over |u|} \big( |\nabla |u| |^2 - | \nabla u |^2 \big) + g\left(r(u), {\frac{u}{ |u|}} \right).$$
Indeed,  we have the following consequence of  the Bochner identity \eqref{Bochner}
$$ g \left( \overrightarrow{ \Delta} u, {u \over |u| }\right) =  {1 \over 2 } { \Delta_g |u|^2\over |u| } - { | \nabla u |^2 \over |u|} = \Delta_g |u| +{ |\nabla |u| |^2 \over |u| } - { | \nabla u |^2 \over |u|}.$$ 
 By the Kato inequality \eqref{katoin}, we have that 
$$  |\nabla |u| |^2 - | \nabla u |^2  \leq 0$$ and thanks to {\bf(H2)}, we also get that  
$$  g\left(r(u), {\frac{u}{ |u|}} \right) \leq - c_0 |u|,$$ therefore, we finally obtain  that
$$   \partial_{t} |u| - \Delta_{g} |u| + c_0|u| \leq 0$$
and the estimate follows from the maximum principle.
\end{proof}

 As a consequence $L^p\rightarrow L^q$ estimates for $(\Delta_{g} - c_0)$ will  imply $L^p \rightarrow L^q$ estimates for
 $ ( \overrightarrow{\Delta} +r ) $. Therefore, we shall first establish the dispersive estimates for the heat equation
 associated to the Laplace Beltrami.
 
  \begin{prop}($L^p \rightarrow L^p $ estimates) \\
  \label{LpLp}
 For every $p \in [1,+ \infty]$, we have for some $c_0, \, \delta_n>0$,   the following
   estimate
   $$ \| e^{t (\Delta_{g} - c_0)} f_{0} \|_{L^p(M)} \leq e^{- t  { \left( \frac{ 4 \delta_n (p-1) }{ p^2} + {c_0}  \right)} \,  }\,\|f_{0}\|_{L^p(M)}.$$
  \end{prop}

\begin{proof}
Let us set $f(t,x)=e^{t (\Delta_{g} -c_0) } f_{0}$, then  $f$  is a solution of 
\begin{equation}\label{HEATPROP}
 \partial_{t} f- (\Delta_{g}-c_0) f = 0.
 \end{equation}
By multiplying the equation by $|f|^{p-2} f$ and by integrating on the manifold, we find
\begin{equation}\label{INEQLP}
{ d \over dt }  \|f\|_{L^p}^p +  4 \, {p-1 \over p}  \| \nabla ( |f|^{p \over 2}) \|_{L^2}^2 + c_0 p   \|f\|_{L^p}^p \leq 0.
\end{equation}
 By using 
 the  Poincar\'e inequality in (4) Remark \ref{REMP},  there is  some  $\delta_n >0$  such that
 $$\delta_n \|h\|_{L^2}^2  \leq  \| \nabla h\|_{L^2}^2$$
  with $ h= |f|^{p \over 2}$, we obtain that 
 $${ d \over dt }  \|f\|_{L^p}^p +  \left( \, 4 {\delta_n (p-1) \over p} + c_0 p \right) \|f\|_{L^p}^p \leq 0, $$
so by a Gronwall type inequality we can conclude.
\end{proof}

 \begin{prop}[$L^1 \rightarrow L^\infty $ estimates] 
 \label{L1Linfty}
 For every $p \in [1,+ \infty]$, we have   the dispersive
   estimate
   $$ \| e^{t (\Delta_{g}-c_0)} f_{0} \|_{L^\infty(M)} \leq  c_n(t) e^{-t (\frac{\delta_n}{2} +c_0)} \|f_{0}\|_{L^1(M)}, \,\, \mathrm{ with }\,\,\,
   c_n(t)=  C_{n} \max{\left({1 \over   t^{n\over 2}}, 1 \right)}.$$
  \end{prop}
We shall give a proof  suitable for any manifold that satisfies our assumptions {\bf(H1-4)}. For non-compact manifolds that enjoy a nice Fourier analysis like the hyperbolic spaces, Damek-Ricci spaces or symmetric spaces of non-compact type, such results can be obtained directly from heat kernel estimates (see \cite{AnkerJi}, \cite{APV1} and others).
  \begin{proof}
  We need to distinguish the $n\geq 3$ and $n=2$ cases due to the fact that the Sobolev embedding  of $H^1(M)$
   in $L^{2^*}(M)$ is critical in dimension $2$.
 \begin{itemize}   
\item
  We begin with the proof of the case of dimension bigger than $3$, which is more direct. We use a classical argument (see for example \cite{Zuazua}) to prove dispersive estimates for the heat equation in euclidean cases by using suitable energy estimates and Sobolev embeddings. Here we can use the Poincar\'e inequality in our argument to improve the large time decay. 
We first use \eqref{INEQLP} with $p=2$ 
$${ d \over dt }  \|f(t)\|_{L^2(M)}^2 +  2 \| \nabla f(t) \|_{L^2(M)}^2 + 2 c_0    \|f(t)\|_{L^2(M)}^2 \leq 0$$
and by combining it with Sobolev-Poincar\'e inequalities  in (2) and (4) Remark \ref{REMP}, we have
$${ d \over dt }  \|f(t)\|_{L^2(M)}^2 +  \eta_n \| f(t) \|_{L^{2^*}(M)}^2 + (\delta_n + 2 c_0)    \|f(t)\|_{L^2(M)}^2 \leq 0.$$
Since by interpolation and the decay of the $L^1$ norm (\eqref{INEQLP} with $p=1$), we have
 $$ \|f(t)\|_{L^2(M)} \leq  \|f(t)\|_{L^1(M)}^{\alpha} \|f(t)\|_{L^{2^*}(M)}^{1-\alpha} \leq C^{\alpha} \|f(t)\|_{L^{2^*}(M)}^{1-\alpha}$$ with $ \frac{1}{2}= \alpha + \frac{(1-\alpha)}{2^*}$ that is to say $\alpha = \frac{2}{n+2}$ and $C= \|f(0)\|_{L^1(M)}$, we obtain
 $${ d \over dt }  \|f(t)\|_{L^2(M)}^2 + \frac{\eta_n}{C^{\frac{2 \alpha}{1-\alpha}}} \| f(t) \|_{L^{2^*}(M)}^{\frac{2 }{1-\alpha}} + (\delta_n + 2 c_0)    \|f(t)\|_{L^2(M)}^2 \leq 0.$$
Next by setting $y(t)= \|f(t)\|_{L^2(M)}^2$, we find the following differential inequality 
$$y'(t) + (\delta_n +2 c_0) y(t) \leq - \frac{\eta_n}{C^{\frac{2 \alpha}{1-\alpha}}} y(t)^{\frac{1 }{1-\alpha}}. $$
 Then
 $$ z(t) =  y(t) e^{ (\delta_n +  2 c_0) t}$$
  solves
  $$ z' (t)  \leq  -  \frac{\eta_n}{C^{\frac{2 \alpha}{1-\alpha}}} z(t)^{\frac{1 }{1-\alpha}}  e^{- (\delta_n +  2 c_0) {\alpha \over 1 - \alpha} t}$$
  and hence by integrating, we obtain
  $$z(t) \leq \left[ \frac{C^{\frac{2 \alpha}{1-\alpha}}}{\eta_n} (\delta_n+ 2 c_0) \right]^{\frac{1-\alpha}{\alpha}} \left(1- e^{-(\delta_n +2 c_0) t \frac{\alpha}{1-\alpha}}  \right)^{\frac{\alpha-1}{\alpha}}, $$
this yields
$$y(t) \leq  \left[ \frac{C^{\frac{2 \alpha}{1-\alpha}}}{\eta_n} (\delta_n+ 2 c_0) \right]^{\frac{1-\alpha}{\alpha}} \left(e^{(\delta_n +2 c_0) t \frac{\alpha}{1-\alpha}} -1 \right)^{\frac{\alpha-1}{\alpha}}. $$
We have thus proved that 
$$ \|f(t)\|_{L^2(M)} \leq  c_n(t)^{\frac{1}{2}} e^{-t (\frac{\delta_n}{2} +c_0)} \|f_{0}\|_{L^1(M)}, \,\, \mathrm{ with }\,\,\,
   c_n(t)=  C_{n} \max{\left({1 \over   t^{n\over 2}}, 1 \right)}.$$
Therefore $e^{\frac{t}{2} (\Delta_{g} - c_0)}: L^1(M) \to L^2(M)$ with norm less than $c_n(\frac{t}{2})^{\frac{1}{2}} e^{-\frac{t}{2} (\frac{\delta_n}{2} +c_0)}$. By duality $e^{\frac{t}{2} (\Delta_{g} - c_0)}: L^2(M) \to L^\infty(M)$ with norm less than $c_n(\frac{t}{2})^{\frac{1}{2}} e^{-\frac{t}{2} (\frac{\delta_n}{2} +c_0)}$. 
Finally, since $ e^{{t} (\Delta_{g} - c_0)} = \, e^{\frac{t}{2} (\Delta_{g} - c_0)} \, e^{\frac{t}{2} (\Delta_{g} - c_0)} : L^1(M) \to L^\infty(M)$, with norm less than $c_n(\frac{t}{2}) e^{-{t} (\frac{\delta_n}{2} +c_0)}$ and
we get the desired dispersive estimate for $n \geq 3$.

\item In dimension $n=2$, we shall first prove the $L^2(M) \to L^\infty(M)$ estimate by using the Nash iteration method (see \cite{Nash}). To do so, we use \eqref{INEQLP} with $p=4$ 
$${ d \over dt }  \|f(t)\|_{L^4(M)}^4 +  3 \| \nabla( |f(t)|^2) \|_{L^2(M)}^2 + 4 c_0    \|f(t)\|_{L^4(M)}^4 \leq 0$$
By multiplying by $t$ the last inequality and by using the Gagliardo-Nirenberg and Poincar\'e inequalities (see (3) and (4) in Remark \ref{REMP}) we have
\begin{equation}\label{INEQNA}
{ d \over dt } \left( t  \|f(t)\|_{L^4(M)}^4 \right) \leq C  \| \nabla f(t) \|_{L^2(M)}^2  \| f(t) \|_{L^2(M)}^2. 
\end{equation}
Since by using \eqref{INEQLP} with $p=2$, we have
$$  \|f(t)\|_{L^2(M)}^2 \leq  \|f(0)\|_{L^2(M)}^2 \quad \mathrm{and} \quad \int_0^t  \|\nabla f(\tau)\|_{L^2(M)}^2 d\tau \leq  \|f(0)\|_{L^2(M)}^2, $$   
we obtain by integrating \eqref{INEQNA} on $[0,t]$ the following estimate
\begin{equation} \label{ITER}
 \|f(t)\|_{L^4(M)}^4 \leq \frac{C}{t}  \|f(0)\|_{L^2(M)}^4.
\end{equation}
We have also for $t>s>0$ 
$$  \|f(t)\|_{L^4(M)} \leq \left(\frac{C}{(t-s)}\right)^{\frac{1}{4}}  \|f(s)\|_{L^2(M)}.$$
Since $|f|^{2^{k-1}}$ is a non-negative sub-solution of the heat equation \eqref{HEATPROP}, we can use \eqref{ITER} with $f$ replaced by  $|f|^{2^{k-1}}.$ 
This yields 
\begin{equation*}
 \|f(t)\|_{L^{2^{k +1}}(M)} \leq \left(\frac{C}{(t-s)^{\frac{1}{4}}}\right)^{\frac{1}{2^{k-1}}}   \|f(s)\|_{L^{2^{k}}(M)}.
\end{equation*}
For every $t>0$, let us set $t_k = t\left(1-\frac{1}{k+1}\right)$; we deduce from the last inequality that
\begin{equation*}
 \|f(t_{k+1})\|_{L^{2^{k +1}}(M)} \leq \left(\frac{C (k+2)^{\frac{1}{2}}}{t^{\frac{1}{4}}}\right)^{\frac{1}{2^{k-1}}}   \|f(t_k)\|_{L^{2^{k}}(M)}.
\end{equation*}
By induction we find 
\begin{equation*}
 \|f(t_{k+1})\|_{L^{2^{k +1}}(M)} \leq \frac{C^{\left(\frac{1}{2^{k-1}}+ \frac{1}{2^{k-2}}+ \dots + {1}\right)} \left[ 
 (k+2)^{\frac{1}{2^{k}}} (k+1)^{\frac{1}{2^{k-1}}}  \cdots 2 \right] }{t^{ \left(\frac{1}{2^{k+1}} + \frac{1}{2^{k}} + \dots + \frac{1}{4} \right)}}  \|f(t_1)\|_{L^{2}(M)}.
\end{equation*}
Since, when $k \to \infty$, we have that $$t^{ \left(\frac{1}{2^{k+1}} + \frac{1}{2^{k}} + \dots + \frac{1}{4} \right)} \to t^{\frac{1}{2}},$$ that  $C^{\left(\frac{1}{2^{k-1}}+ \frac{1}{2^{k-2}}+ \dots + {1}\right)} $ and the product $  \left[ (k+2)^{\frac{1}{2^{k}}} (k+1)^{\frac{1}{2^{k-1}}}  \cdots 2  \right]$ are bounded, we get 
\begin{equation}\label{DISP2}
 \|f(t)\|_{L^\infty(M)}\leq \frac{C}{t^{\frac{1}{2}}}  \|f(0)\|_{L^2(M)}.
\end{equation}
As expected on a non-compact manifold with negative curvature, we can improve the decay in the last estimate for large times. \\
Actually, for $t>1$ by the semigroup property we can write 
$$ e^{{t} (\Delta_{g} - c_0)} = \, e^{\frac{t-1}{2} (\Delta_{g} - c_0)} \, e^{ (\Delta_{g} - c_0)} \, e^{\frac{t-1}{2} (\Delta_{g} - c_0)}. $$ 
Thanks to Proposition \ref{LpLp} we have that $e^{\frac{t-1}{2} (\Delta_{g} - c_0)}: L^2(M) \to L^2(M)$ is bounded with norm less than $e^{-\frac{t-1}{2} ({\delta_2} +c_0)}$ and that $e^{\frac{t-1}{2} (\Delta_{g} - c_0)}: L^\infty(M) \to L^\infty(M)$  is bounded with norm less than $e^{- c_0 \frac{t-1}{2} }$. 
Moreover, by \eqref{DISP2} with $t=1$, we have also that $e^{(\Delta_{g} - c_0)}: L^2(M) \to L^\infty(M)$ is bounded. 
Thus for $t>1$ we obtain 
\begin{equation}\label{DISP2Large}
 \|f(t)\|_{L^\infty(M)}\leq {C} e^{- t( \frac{\delta_2}{2} + c_0)}  \|f(0)\|_{L^2(M)}.
\end{equation}
From \eqref{DISP2}  and \eqref{DISP2Large} we get 
$$  \| f(t) \|_{L^\infty(M)} \leq  c_2(t)^{\frac{1}{2}} e^{-t (\frac{\delta_2}{2} +c_0)} \|f_{0}\|_{L^2(M)}, \,\, \mathrm{ with }\,\,\,
   c_2(t)=  C_{2} \max{\left({1 \over   t}, 1 \right)}.$$
As before, by a duality and composition argument we deduce the claimed dispersive estimate in dimension $2$
\begin{equation*}
 \|f(t)\|_{L^\infty(M)}\leq  c_2(t) e^{-t (\frac{\delta_2}{2} +c_0)}   \|f(0)\|_{L^1(M)}.
\end{equation*}

 \end{itemize} 
  \end{proof}

\noindent { \bf{End of the proof of Theorem \ref{theolin1}}.}
  Finally we prove $L^p \rightarrow L^q $ dispersive estimates for the Bochner heat equation.
  Thanks to Lemma \ref{comparaison}, it suffices to prove the corresponding estimates for the 
   Laplace Beltrami semi-group.
  To do so,  we shall use many interpolation arguments.
    First,  we can use Proposition \ref{LpLp} for $p=1$  and Proposition \ref{L1Linfty} to obtain the following estimate
    \begin{equation*}
\;\|\,e^{t (\Delta_g - c_0)} \,\|
_{L^1\to L^r} \leq   C e^{-c_0 t } e^{- \frac{ t \delta_n}{2}(1 - { 1 \over r }) } c_n(t)^{ (1 - { 1 \over r })}
\quad\forall\; r \in [1 ,+ \infty]\,\\
\end{equation*}
and  by duality we deduce that
    \begin{equation*}
\;\|\,e^{t (\Delta_g - c_0)}\,\|
_{L^p\to L^\infty} \leq   C  e^{- c_0 t } e^{- \frac{ t \delta_n}{2 p}}  c_n(t)^{ { 1 \over p }}
\quad\forall\; p \in [1 ,+ \infty]\,.
\end{equation*}
Next,  by interpolating the last estimate and  the $L^p \to L^p$ estimate in Proposition \ref{LpLp}  for $p\in[1, \infty] $, we conclude the proof  obtaining $L^p \to L^q$ estimates for $1\leq p\leq q \leq \infty $ with the norm $c_n(t)^{ \left({1\over p} - { 1 \over q } \right )}\, e^{-c_0 t} \, e^{-  t \,    \gamma_{p,q} } $, where   \, $ \gamma_{p,q}=  \frac{\delta_n}{2} \left[ \left( {1 \over p } - {1 \over q} \right) + {8 \over q} \left( 1 - { 1 \over p } \right)\right]$ and $ c_n(t)=  C_{n} \max{\left({1 \over   t^{n\over 2}}, 1 \right)}$.   
   
\subsection{Proof of Theorem \ref{theolin2}}
We shall split the proof into several Lemmas.
\begin{lemma}\label{lemmaPP}
 Assuming {\bf (H1-4)}, we have the following estimate for $p\geq 2$
 \begin{equation}\label{estimPP}
   \|  \nabla e^{t (\overrightarrow{\Delta} - I) } u_{0} \|_{L^p} \leq C \, \max \left( {1 \over \sqrt{t} }, 1\right)    \|u_{0}\|_{L^p}, \quad \forall \,\,t>0.
 \end{equation}
\end{lemma}
\begin{proof}
The following proof generalizes   to the vectorial Laplacian some arguments yielding log Sobolev inequalities
 for the Laplace Beltrami operator on Riemannian manifolds (see for example \cite{Bakry}). 
Let us consider $ P_{t}= e^{t \Delta_{g} }$ and  $Q_{t}= e^{t (\overrightarrow{\Delta} - I )} $.
 We shall prove the crucial pointwise estimate
\begin{equation} 
\label{bakry} | \nabla Q_{t} u_{0} |^2 \leq  {1 \over d(t)} \left(P_{t}( |u_{0}|^2) - |Q_{t} u_{0}|^2 \right) - |Q_{t} u_{0}|^2, \quad  {1 \over d(t)}=  { \alpha_1 \over (e^{2 \alpha_{1} t } - 1)},\end{equation}
with $\alpha_1 = -\left( \max(c_0, \frac{1}{c_0}) + 2 K n + 2 K n^{\frac{3}{2}} + 2 \right) <0. $\\
Since $${1 \over d(t)}=  { \alpha_1 \over {(e^{2 \alpha_{1} t } - 1)}} \leq C_1 \,
 \left \{  \begin{array}{ll} \frac{1}{t} \quad \mathrm{if} \,\, 0<t\leq 1, \\  \alpha_1  \quad \mathrm{if} \,\, t\geq 1,  \end{array} \right. $$
we obtain 
$$ | \nabla Q_{t} u_{0} |^2 \leq C_1  \max \left( \frac{1}{t}, \alpha_1 \right)   P_{t}( |u_{0}|^2) $$
which,  by integrating on $M$ and by using that $P_{t}: \,  L^{p \over 2 }\rightarrow  L^{p\over 2}$ is  bounded  for $p \geq 2$ proved in Proposition \ref{LpLp}, implies 
\begin{align}
\nonumber  \| \nabla Q_{t} u_{0} \|_{L^p}&  \leq C_{n}  \, \max \left( {1 \over \sqrt{t} }, 1\right)   \|P_{t} (|u_{0}|^2 ) \|_{L^{p \over 2}}^{1 \over 2} \\
  \label{SEPiccolo}& \leq  C_{n}  \,  \max \left( {1 \over \sqrt{t} }, 1\right)    e^{ -   { 4t \delta_n \,  (p-2) \over p^2 }  } \|u_{0}\|_{L^p}
\end{align}
which yields the proof of the Lemma.\\
It remains to prove \eqref{bakry}. We note that by using the following properties
\begin{equation}\label{propsem}
{ d \over ds } P_{s} = \Delta_{g}  P_{s} = P_{s}  \Delta_{g} \,\,\, \mathrm{and} \,\,\, \, { d \over ds } Q_{{t-s}} = - (\overrightarrow{\Delta } -I)  Q_{{t-s}},  
\end{equation}
we can write
\begin{align*}
  P_{t}\left( |u_{0}|^2 \right) - \left| Q_{t} \, u_{0}\right |^2 &= \int_{0}^t { d \over ds } \left( P_{s}\left( |Q_{t-s} \, u_{0}|^2 \right) \right) \, ds \\ &= \int_{0}^t P_{s} \left[ \left( \Delta_{g}  | Q_{{t-s}} \,u_{0} |^2   \right) - 2 g \left( \overrightarrow{\Delta}  Q_{t-s} u_{0},  Q_{t-s} \,u_{0} \right) + 
  2 |Q_{t-s} u_0|^2 \right] ds.
  \end{align*}
From  Bochner's  identity  \eqref{Bochner} in Lemma \ref{bochner} for the vector field $u = Q_{t-s}\, u_0,$ we obtain
 $$  P_{t}\left( |u_{0}|^2 \right) - \left| Q_{t} u_{0}\right |^2= 2 \int_{0}^t P_{s} \left( | \nabla Q_{t-s} u_{0}|^2 + |Q_{t-s} u_0|^2 \right)\, ds:= 2\int_{0}^t  e^{2 \alpha s}  \psi(s) \, ds.$$
 where
 $$ \psi(s) = e^{-2 \alpha  s}  \, P_{s} \left( | \nabla Q_{t-s} u_{0}|^2 + |Q_{t-s} u_0|^2 \right) \geq 0.$$
 If we can  choose the parameter  $\alpha \in \mathbb{R}$  such that $ \psi$ is nondecreasing, we  obtain
 $$ P_{t}\left( |u_{0}|^2 \right) - \left| Q_{t} \, u_{0}\right |^2\geq 2 \, \psi (0) \int_{0}^t e^{2 \alpha s}\, ds =  d(t)\, \left( | \nabla Q_{t} \, u_{0}|^2 + |Q_{t-s} u_0|^2 \right)$$ 
 and the conclusion follows.
 Finally we have to prove that: there exists $\alpha \in \mathbb{R}$  such that $ \psi'(s)\geq0$. 
  With  explicit computations and  by using again the semigroup properties  \eqref{propsem},
 we write for all $m \in M$
\begin{multline*}
 \psi'(s)_{/m}  = e^{-2 \alpha  s}  \, P_{s} \left[ -2 \alpha \left(  | \nabla Q_{t-s} u_{0}|^2 + |Q_{t-s} u_0|^2  \right) + \Delta_g  \left(  | \nabla Q_{t-s} u_{0}|^2 + |Q_{t-s} u_0|^2  \right) \right. \\ \left. - 2 g\left(\nabla ( \overrightarrow{\Delta} - I)  Q_{t-s} u_{0},  \nabla  Q_{t-s} \,u_{0} \right)  + g\left( ( \overrightarrow{\Delta} - I)  Q_{t-s} u_{0},   Q_{t-s} \,u_{0} \right)  \right]_{/m},
 \end{multline*}
by using  Bochner's  identity  \eqref{Bochner} again, we can simplify the last expression obtaining
$$ \psi'(s)_{/m}  = e^{-2 \alpha  s}  \, P_{s} \left[ B(m) \right] $$
 where
 \begin{multline*}
 B(m)= \left[ \left( \Delta_{g}  | \nabla Q_{{t-s}} \,u_{0} |^2   \right) - 2 g \left( \nabla  \overrightarrow{\Delta}   Q_{t-s} u_{0},  \nabla  Q_{t-s} \,u_{0} \right) 
  \right. 
  \\ \Bigl.- (2 \alpha+1) | Q_{{t-s}} \,u_{0} |^2 - 2 \alpha  | \nabla Q_{{t-s}} \,u_{0} |^2 \Bigr]_{/m}. 
  \end{multline*}
 By the maximum principle it is sufficient to prove that $B(m) \geq 0, \, \forall \, m \in M$.
We compute $B(m)$ for each $m \in M$ by using normal geodesic coordinates at $m$.
Let us set $e_{i}= \partial/\partial x^i= \partial_{i}$,  then $(e_{1}, \cdots, e_{n})$ is an orthonormal basis at $m$. By using the properties  \eqref{normal} of the normal coordinates at $m$ and by the connection property \eqref{levi}, we can write the first term of $B(m)$ as follows
\begin{multline*}
 \left( \Delta_{g}  | \nabla Q_{{t-s}} \,u_{0} |^2   \right)_{/m} = \partial^2_{k} \left(g^{ij} 
 g\left(  \nabla_{e_{i}}(Q_{{t-s}} \,u_{0}), \nabla_{e_{j}}(Q_{{t-s}} \,u_{0}) \right)\right)_{/m}  \\
= \left(  \partial^2_{k} g^{ij} \right)_{/m}
 g\left(  \nabla_{e_{i}}(Q_{{t-s}} \,u_{0}), \nabla_{e_{j}}(Q_{{t-s}} \,u_{0}) \right)_{/m} + 
 g^{ij}_{/m}  \partial^2_{k} \left(
 g\left(  \nabla_{e_{i}}(Q_{{t-s}} \,u_{0}), \nabla_{e_{j}}(Q_{{t-s}} \,u_{0}) \right)\right)_{/m} \\
= \left(  \partial^2_{k} g^{ij} \right)_{/m}
 g\left(  \nabla_{e_{i}}(Q_{{t-s}} \,u_{0}), \nabla_{e_{j}}(Q_{{t-s}} \,u_{0}) \right)_{/m} 
 +  2 g^{ij}_{/m}  
 g\left(  \nabla_{e_{k}}\nabla_{e_{i}}(Q_{{t-s}} \,u_{0}), \nabla_{e_{k}} \nabla_{e_{j}}(Q_{{t-s}} \,u_{0}) \right)_{/m} \\  +  g^{ij}_{/m}  g\left( \nabla_{e_{k}}   \nabla_{e_{k}}\nabla_{e_{i}}(Q_{{t-s}} \,u_{0}), \nabla_{e_{j}}(Q_{{t-s}} \,u_{0}) \right)_{/m} +  g^{ij}_{/m}  g\left(\nabla_{e_{i}}(Q_{{t-s}} \,u_{0}),  \nabla_{e_{k}}   \nabla_{e_{k}} \nabla_{e_{j}}(Q_{{t-s}} \,u_{0}) \right)_{/m}. 
 \end{multline*}
Thus, by using the expression \eqref{bochn} of the Bochner Laplacian and the norm on tensors \eqref{norm2} in normal coordinates at $m$, we can write
\begin{multline*}
 \left( \Delta_{g}  | \nabla Q_{{t-s}} \,u_{0} |^2   \right)_{/m} = 
 \left(  \partial^2_{k} g^{ij} \right)_{/m}
 g\left(  \nabla_{e_{i}}(Q_{{t-s}} \,u_{0}), \nabla_{e_{j}}(Q_{{t-s}} \,u_{0}) \right)_{/m}  \\
+  2  | \nabla^2 Q_{t-s} u_{0} |_{/m}^2 + 2 g \left(   \overrightarrow{\Delta}  \nabla_{e_{i}}(Q_{{t-s}} \,u_{0}),  \nabla_{e_{i}}(Q_{{t-s}} \,u_{0}) \right)_{/m}.
  \end{multline*}
Therefore 
\begin{multline*}
B(m)=
 \left(  \partial^2_{k} g^{ij} \right)_{/m}
 g\left(  \nabla_{e_{i}}(Q_{{t-s}} \,u_{0}), \nabla_{e_{j}}(Q_{{t-s}} \,u_{0}) \right)_{/m} + 2  | \nabla^2 Q_{t-s} u_{0} |_{/m}^2   \\
  + 2 g \left(   \overrightarrow{\Delta}  \nabla_{e_{i}} (Q_{{t-s}} \,u_{0}) -  \nabla_{e_{i}}  \overrightarrow{\Delta} (Q_{{t-s}} \,u_{0}),  \nabla_{e_{i}}(Q_{{t-s}} \,u_{0}) \right)_{/m} \\  - (2 \alpha+1) | Q_{{t-s}} \,u_{0} |_{/m}^2 -  2 \alpha  | \nabla Q_{{t-s}} \,u_{0} |_{/m}^2.
   \end{multline*}
Now, let us compute $  \left(\overrightarrow{\Delta}  \nabla_{e_{i}} (Q_{{t-s}} \,u_{0}) -  
\nabla_{e_{i}}  \overrightarrow{\Delta} (Q_{{t-s}} \,u_{0})\right)$ at $m$. By using \eqref{bochn} and \eqref{tensorecurv}, we have
\begin{multline*}
\nabla_{e_{i}}  \overrightarrow{\Delta} (Q_{{t-s}} \,u_{0})_{/m} = \left( \nabla_{e_{i}} (\nabla_{e_{k}} \nabla_{e_{k}}  (Q_{{t-s}} \,u_{0}) -   \nabla_{ \nabla_{e_{k}} e_{k}} (Q_{{t-s}} \,u_{0})) \right)_{/m} \\
= \left[ \nabla_{e_{k}} \nabla_{e_{i}} (\nabla_{e_{k}}  (Q_{{t-s}} \,u_{0})) - R(e_i, e_k) \nabla_{e_{k}} (Q_{{t-s}} \,u_{0}) + \nabla_{[e_{i},e_{k}]} \nabla_{e_{k}} (Q_{{t-s}} \,u_{0}) - \nabla_{e_{i}}   \nabla_{ \nabla_{e_{k}} e_{k}} (Q_{{t-s}} \,u_{0})  \right]_{/m}
    \end{multline*}
 We note that $[e_{i},e_{k}] =0$ for every point $p$ in a vicinity of $m$ and since $ \nabla_{ \nabla_{e_{k}} e_{k}} =0$ at $m$ this yields
 \begin{multline*}
  \nabla_{e_{i}}   \nabla_{ \nabla_{e_{k}} e_{k}} (Q_{{t-s}} \,u_{0})_{/m}   \\=  \left[   \nabla_{ \nabla_{e_{k}} e_{k}} ( \nabla_{e_{i}} Q_{{t-s}} \,u_{0}) - R(e_i,  \nabla_{ \nabla_{e_{k}} e_{k}}) Q_{{t-s}} \,u_{0} + \nabla_{[e_{i}, \nabla_{ \nabla_{e_{k}} e_{k}} ]}(Q_{{t-s}} \,u_{0}) \right]_{/m} = 
   \nabla_{[e_{i}, \nabla_{ \nabla_{e_{k}} e_{k}} ]}(Q_{{t-s}} \,u_{0})_{/m}  
   \end{multline*}
and applying again \eqref{tensorecurv}, we obtain
 \begin{multline*}
\nabla_{e_{i}}  \overrightarrow{\Delta} (Q_{{t-s}} \,u_{0})_{/m} = \\ \left[ \nabla_{e_{k}}\nabla_{e_{k}}\nabla_{e_{i}} (Q_{{t-s}} \,u_{0})  - \nabla_{e_{k}} \left(R(e_i,  e_k)Q_{{t-s}} \,u_{0}\right) - R(e_i,  e_k) \nabla_{e_{k}} (Q_{{t-s}} \,u_{0}) -\nabla_{[e_{i}, \nabla_{ \nabla_{e_{k}} e_{k}} ]}(Q_{{t-s}} \,u_{0}) \right]_{/m} \\ =  \overrightarrow{\Delta} \nabla_{e_{i}} (Q_{{t-s}} \,u_{0})_{/m} - [( \nabla R )(e_k,e_i,e_k)](Q_{{t-s}} \,u_{0})_{/m} \\Ê- 
2 R(e_i,e_k) \nabla_{e_{k}} (Q_{{t-s}} \,u_{0})_{/m}  -   \nabla_{[e_{i}, \nabla_{ \nabla_{e_{k}} e_{k}} ]}(Q_{{t-s}} \,u_{0})_{/m}. 
 \end{multline*}
 By using the Cristoffel symbols, we have
 $$[e_{i}, \nabla_{ \nabla_{e_{k}} e_{k}} ]_{/m} = \left[{\left(\partial_i \Gamma_{kk}^l \right) e_l}\right]_{/m}.$$
 Thus 
 \begin{multline*}
B(m)=
 \left(  \partial^2_{k} g^{ij} \right)_{/m}
 g\left(  \nabla_{e_{i}}(Q_{{t-s}} \,u_{0}), \nabla_{e_{j}}(Q_{{t-s}} \,u_{0}) \right)_{/m} + 2  | \nabla^2 Q_{t-s} u_{0} |_{/m}^2   \\ +
  2 g \left(   [( \nabla R )(e_k,e_i,e_k)](Q_{{t-s}} \,u_{0}), \nabla_{e_{i}} (Q_{{t-s}} \,u_{0})   \right)_{/m} +  
  4 g \left( R(e_i,e_k) \nabla_{e_{k}} (Q_{{t-s}} \,u_{0}), \nabla_{e_{i}} (Q_{{t-s}} \,u_{0})   \right)_{/m}  \\ +  2 \left(\partial_i \Gamma_{kk}^l \right)_{/m}  g\left(  \nabla_{e_{l}}(Q_{{t-s}} \,u_{0}), \nabla_{e_{i}}(Q_{{t-s}} \,u_{0}) \right)_{/m}   - (2 \alpha+1) | Q_{{t-s}} \,u_{0} |_{/m}^2 - 2 \alpha  | \nabla Q_{{t-s}} \,u_{0} |_{/m}^2.
   \end{multline*}
 We can rewrite the last expression with the curvature tensors. Indeed by \eqref{tensRiem}, we have 
 $$ 4 g \left( R(e_i,e_k) \nabla_{e_{k}} (Q_{{t-s}} \,u_{0}), \nabla_{e_{i}} (Q_{{t-s}} \,u_{0})   \right)_{/m} = 
 4 \Riem(e_{i}, e_{k}, \nabla_{e_{k}} (Q_{t-s} u_{0}), \nabla_{e_{i}} (Q_{t-s} u_{0}))_{/m} $$
 By \cite{Druet} (p.15) and \eqref{tensRic} 
 $$\sum_{k=1}^n 
 \left(  \partial^2_{k} g^{ij} \right)_{/m} = \frac{2}{3} \sum_{k=1}^n \Riem(e_{i}, e_{k}, {e_{j}}, e_k)_{/m} =  \frac{2}{3} \Ric(e_i,e_j)_{/m}. $$
 By using $$\Gamma_{ij}^k= \frac{1}{2} \left(g^{kl} \left( \partial_i g_{jl} + \partial_j g_{il} - \partial_l g_{ij} \right) \right)$$
 and again by \cite{Druet}, we also deduce 
  $$\sum_{k=1}^n 
  \left(  \partial_{i} \Gamma_{kk}^l  \right)_{/m} =  \sum_{k=1}^n \left( -\frac{1}{3}\Riem(e_{k}, e_{i}, {e_{l}}, e_k)  + \frac{1}{6}\Riem(e_{k}, e_{l}, {e_{k}}, e_i) + \frac{1}{6}\Riem(e_{k}, e_{i}, {e_{k}}, e_l) \right)_{/m},$$
  by the symmetry properties of Riemann tensor and  \eqref{tensRiem}, we obtain
   $$\sum_{k=1}^n \left(  \partial_{i} \Gamma_{kk}^l  \right)_{/m} =  \frac{2}{3}  \Ric(e_i,e_l)_{/m}.$$
   Thus we have
 \begin{multline*}
B(m)=   2 \Ric (e_{i}, e_{j})_{/m} \, g( \nabla_{e_{i}} (Q_{t-s} u_{0}), \nabla_{e_{j}} (Q_{t-s} u_{0}))_{/m} + 2  | \nabla^2 Q_{t-s} u_{0} |_{/m}^2 + \\  4 \Riem(e_{i}, e_{k}, \nabla_{e_{k}} (Q_{t-s} u_{0}), \nabla_{e_{i}} (Q_{t-s} u_{0}))_{/m} - 2 \alpha | \nabla Q_{{t-s}} \,u_{0} |^2_{/m}  +\\
2 g \left(   [( \nabla R )(e_k,e_i,e_k)](Q_{{t-s}} \,u_{0}), \nabla_{e_{i}} (Q_{{t-s}} \,u_{0})   \right)_{/m}  - (2 \alpha+1) | Q_{{t-s}} \,u_{0} |_{/m}^2. \end{multline*}
 By hypothesis  {\bf (H1-2)} we have 
 $$B(m) \geq   | \nabla Q_{t-s} u_{0} |_{/m}^2 \left(-2 \alpha -2 \max \left(c_0,\frac{1}{c_0}\right) - 4Kn -2K n^{\frac{3}{2}} \right) +   |  Q_{t-s} u_{0} |_{/m}^2 \left(-2 \alpha -1 - 2K n^{\frac{3}{2}} \right).$$
 We can see that $B(m)$ is positive for $\alpha \leq \alpha_1.$
 In particular, by choosing 
 $$\alpha_1= -\left(\max \left(c_0,\frac{1}{c_0}\right) +2Kn+2K n^{\frac{3}{2}}+2\right),$$ we end the proof of the Lemma.
\end{proof}
\begin{lemma}\label{lemmaPPcompl}
 Assuming {\bf (H1-4)}, we have the following estimate for $1<p<+ \infty$
 \begin{equation}\label{estimPPcompl}
   \|  \nabla e^{t (\overrightarrow{\Delta} - I) } u_{0} \|_{L^p} \leq C \, \max \left( {1 \over \sqrt{t} }, 1\right)    \|u_{0}\|_{L^p}, \quad \forall \,\, t>0.
 \end{equation}
\end{lemma}
\begin{proof}
We have only to prove the estimate \eqref{estimPPcompl}  for $ 1<p \leq 2$. We can use that,
 for vector fields $u\in \Gamma(TM),$  $$ \|\nabla  u  \|_{L^p} \sim  \| (- \overrightarrow{\Delta})^{1 \over 2 } u \|_{L^p}, \quad 1<p<+ \infty,$$ which is essentially the $L^p$
  boundedness of the Riesz transform $\nabla  (- \overrightarrow{\Delta})^{-{1 \over 2 }}    $ (see \cite{Lohoue}, \cite{Strichartz}). Thus $ (- \overrightarrow{\Delta})^{1 \over 2 }  \, e^{ t (\overrightarrow{\Delta} -I)} : L^p  \rightarrow L^p$ satisfies \eqref{estimPP}  for $2 \leq  p <+\infty$ and 
  by duality  $ \left( (- \overrightarrow{\Delta})^{1 \over 2 }  e^{ t (\overrightarrow{\Delta} -I)}\right)^* =  e^{ t \overrightarrow{(\Delta}-I) }
  (- \overrightarrow{\Delta})^{1 \over 2 } =   (- \overrightarrow{\Delta})^{1 \over 2 }  e^{ t (\overrightarrow{\Delta}-I) } :L^{p'}  \rightarrow L^{p'},$
   we obtain \eqref{estimPPcompl}   for $ 1 <p' \leq 2$. 
\end{proof}
 We shall now establish short time $L^p \to L^p$ estimates for the operator $\nabla e^{ t (\overrightarrow{\Delta} +r)}.$
 \begin{lemma}\label{lemmaPPshorttime}
 Assuming {\bf (H1-4)}, we have the following estimate for $1<p<+ \infty$
 \begin{equation}\label{estimPPshorttime}
   \|  \nabla e^{t (\overrightarrow{\Delta} +r )} u_{0} \|_{L^p} \leq C \, \max \left( {1 \over \sqrt{t} }, 1\right)    \|u_{0}\|_{L^p}, \quad \forall \,\, 0<t\leq1.
 \end{equation}
 \end{lemma}
\begin{proof}
 Since $$\nabla e^{t \overrightarrow{\Delta} } = e^t \, \nabla e^{t (\overrightarrow{\Delta} -I)},$$
 we deduce from Lemma \ref{lemmaPPcompl} that
  \begin{equation}
   \|  \nabla e^{t \overrightarrow{\Delta} } u_{0} \|_{L^p} \leq C \, \max \left( {1 \over \sqrt{t} }, 1\right)    \|u_{0}\|_{L^p}
 \end{equation}
 for $ 0<t\leq1$ and  $1<p<+ \infty$.
 By using the Duhamel formula, we have
 $$ \nabla e^{t (\overrightarrow{\Delta} +r )} u_{0} = \nabla e^{t \overrightarrow{\Delta} } u_{0} +
  \int_0^t \nabla e^{(t-s) \overrightarrow{\Delta} }\left(r \,(e^{s(\overrightarrow{\Delta} +r )} u_{0})\right) ds,$$
 thus we obtain
  $$   \| \nabla e^{t (\overrightarrow{\Delta} +r )} u_{0}  \|_{L^p} \leq  \| \nabla e^{t \overrightarrow{\Delta} } u_{0}  \|_{L^p} + 
  \int_0^t  \| \nabla e^{(t-s) \overrightarrow{\Delta} }\left(r \,(e^{s(\overrightarrow{\Delta} +r )} u_{0})\right)   \|_{L^p} ds$$
 hence
 $$\leq C \, \max \left( {1 \over \sqrt{t} }, 1\right)    \|u_{0}\|_{L^p} + C \int_0^t \max \left( {1 \over \sqrt{t-s} }, 1\right) 
   \| r \,(e^{s(\overrightarrow{\Delta} +r )} u_{0})  \|_{L^p} ds. $$
 By hypothesis  {\bf (H2)} and \eqref{dispersiveestimates}, we have
  \begin{equation*}
  \| r \,(e^{s(\overrightarrow{\Delta} +r )} u_{0})  \|_{L^p} \leq C \max\left(c_0,\frac{1}{c_0}\right) 
\| e^{t ( \overrightarrow{\Delta} + r)} u_{0} \|_{L^p} \leq C \max\left(c_0,\frac{1}{c_0}\right) \, e^{-  t  ( \gamma_{p,q} + c_0) } \, \| u_{0} \|_{L^p}.
\end{equation*}
 Thus we obtain 
 \begin{equation*}
   \|  \nabla e^{t (\overrightarrow{\Delta} +r )} u_{0} \|_{L^p} \leq C \, \max \left( {1 \over \sqrt{t} }, 1\right)    \|u_{0}\|_{L^p}
   \end{equation*}
    for short time $0<t \leq 1$ and $1<p<+ \infty.$
 \end{proof}
By previous $L^p \to L^p$ estimates \eqref{estimPPshorttime} for the operator $\nabla  e^{t (\overrightarrow{\Delta} +r )}$ and by duality argument, the following estimates are true for $1<p<+\infty$ and short time $0<t \leq 1$ 
 \begin{equation}\label{eststarshort}
   \| e^{t (\overrightarrow{\Delta} +r) } \nabla^{*}  T_{0} \|_{L^p} \leq C \, \max \left( {1 \over \sqrt{t} }, 1\right)    \|T_{0}\|_{L^p}, 
 \end{equation}
 for tensors $T_0 \in L^p(\Gamma(TM \otimes T¬^{*}M)).$
 To get large time estimates, we use the semi-group property combined with  last estimates at $t=1$ and $L^p \to L^p$ dispersive estimates \eqref{dispersiveestimates}
 $$e^{t (\overrightarrow{\Delta} +r) } \nabla^{*}= e^{(t-1) (\overrightarrow{\Delta} +r) } \left(e^{(\overrightarrow{\Delta} +r) }  \nabla^{*} \right): L^p \rightarrow L^p \rightarrow L^p.$$  
This yields for $1<p<+ \infty$ and $t \geq 1$
 \begin{equation}\label{eststarlarge}
   \| e^{t (\overrightarrow{\Delta} +r) } \nabla^{*}  T_{0} \|_{L^p} \leq C \, e^{-t \,\left(c_0+ {4 \delta_n \over p} \left( 1 - { 1 \over p } \right)\right)}    \|T_{0}\|_{L^p},
 \end{equation}
  for all $T_0 \in L^p(\Gamma(TM \otimes T¬^{*}M)).$
From \eqref{eststarshort}  and \eqref{eststarlarge} we deduce for all time $t>0$ and $1<p<+ \infty$
 \begin{equation}
   \| e^{t (\overrightarrow{\Delta} +r) } \nabla^{*}  T_{0} \|_{L^p} \leq C \,  \max \left( {1 \over \sqrt{t} }, 1\right)  \, e^{-t \,\left(c_0+ {4 \delta_n \over p} \left( 1 - { 1 \over p } \right)\right)}   \|T_{0}\|_{L^p}.
 \end{equation}
   for all $T_0 \in L^p(\Gamma(TM \otimes T¬^{*}M)).$
 Finally, by duality again we finish the proof of Theorem \ref{theolin2} obtaining   smoothing estimates \eqref{smoothingestimates}  for vector fields.
 
 \subsection{The case of the Stokes equations  }
 In this section, we shall consider the following Stokes type linear equations
 \begin{equation}\label{NSprew1}
 \left\{ \begin{array}{ll}   \partial_t u  -  \overrightarrow{\Delta} u - r(u)  + 2 \grad  (-\Delta_g)^{-1}\dive (ru) = 0 ,\\  \dive {u_0}= 0, \\   u_{|t=0}= u_{0}, \quad u_0 \in \Gamma(TM).
 \end{array} \right.
  \end{equation}
 It will be convenient to use the linear operator
 $$ Bu = - 2 \grad  (-\Delta_g)^{-1}\dive (ru).$$
 Note that thanks to the boundedness of the Riesz transform on a manifold that satisfies the assumptions {\bf(H1-4)}
  (see again \cite{Lohoue}), $B$ is bounded as a linear operator  $L^p \to L^p$ for every $p$, $ 1<p<+\infty$.
  First we shall prove the following dispersive estimates for small time:
\begin{prop}
 \label{propB1}
Assuming {\bf (H1-4)},  the solution of \eqref{NSprew1} satisfies the following dispersive estimates:
 for every   $p,q$ such that $ 1 < p \leq q <  +\infty, $ there exists $C>0$ such that 
\begin{equation}\label{dispersiveestimatesBSmall}
\|  u(t) \|_{L^q} \leq C c_n(t)^{ \left({1\over p} - { 1 \over q } \right )}  \, \| u_{0} \|_{L^p} \quad \forall \,\, 0< t \leq 2, \,\,\forall \, u_{0} \in L^p(\Gamma(TM)),
\end{equation}
 with  $ c_n(t)=  \max \big({1 \over  t^{n\over 2}}, 1 \big)$.
 \end{prop}
\begin{proof}
We begin the proof with the case $p=q$. 
By using the Duhamel formula we can write 
$$u(t)= e^{t ( \overrightarrow{\Delta} + r)} u_0 + \int_0 ^t  e^{(t -\tau) ( \overrightarrow{\Delta} + r)} Bu(\tau) \, d\tau.$$
Thanks to  \eqref{dispersiveestimates} in Theorem \ref{theolin1} with $p=q$, we have 
$$
\|u(t)\|_{L^q} \leq C \|u_0\|_{L^q} +  \int_0 ^{t} C \|u(\tau) \|_{L^q} d \tau 
$$
and hence from the Gronwall inequality, we find
\begin{equation}
\label{LqestB}
\|u(t)\|_{L^q} \leq C  e^{C t}\|u_0\|_{L^q}, \quad \forall t\geq0.
\end{equation}
Note that the large time behavior is not  good and will be improved later.\\
 Next, thanks to \eqref{dispersiveestimates} in Theorem \ref{theolin1}, we obtain
 \begin{align*}
  \|u(t)\|_{L^q}  & \leq c_n(t)^{ \left({1\over p} - { 1 \over q } \right )} e^{-  t  ( \gamma_{p,q} + c_0) }  \|u_0\|_{L^p}  +  \\
&  \qquad  \int_0 ^{t/2} \left[c_n(t- \tau)\right]^{ \left({1\over p} - { 1 \over q } \right )} e^{-  (t- \tau)  ( \gamma_{p,q} + c_0) }  \|u(\tau)\|_{L^p}  d\tau + C  \int_{t/2}^t  \|u(\tau) \|_{L^q} d \tau \\
 &  \leq  Cc_n(t)^{ \left({1\over p} - { 1 \over q } \right )}   \|u_0\|_{L^p}  +     C c_n(t)^{ \left({1\over p} - { 1 \over q } \right )}  \int_0 ^{t/2}   \|u(\tau)\|_{L^p}  d\tau + C  \int_{t/2}^t  \|u(\tau) \|_{L^q} d \tau
  \end{align*}
by using  the  estimate \eqref{LqestB}, we deduce the following estimate for  $0<t \leq 2$
$$ \|u(t)\|_{L^q}  \leq  Cc_n(t)^{ \left({1\over p} - { 1 \over q } \right )}   \|u_0\|_{L^p}  +  
    C  \int_{t/2}^t  \|u(\tau) \|_{L^q} d \tau $$
     and hence
by setting $y(t) =  c_n(t)^{- \left({1\over p} - { 1 \over q } \right )}   \|u(t)\|_{L^q}$,  
we get
$$ y(t) \leq C   \|u_0\|_{L^p}  + {C \over  c_n(t)^{ \left({1\over p} - { 1 \over q } \right )}} \int_{t\over 2}^t   
c_n(\tau)^{ \left({1\over p} - { 1 \over q } \right )} y(\tau)\, d\tau \leq C   \|u_0\|_{L^p} + C \int_0^t  y(\tau)\, d\tau$$
 and by  the Gronwall inequality we can conclude.
\end{proof}

 \begin{theorem}
\label{theolinB1}
Assuming {\bf (H1-4)},  there exist $\beta \geq c_0 >0$ such that the solution of  the Cauchy problem \eqref{NSprew1} satisfies the following  estimates 
\begin{equation}\label{BLpestimates}
\| u(t)  \|_{L^p} \leq \, C e^{-  \beta t } \, \left(\| u_{0} \|_{L^p}+ \| u_{0} \|_{L^2}\right), \quad \forall \,\,t>0,
\end{equation}
for every   $p$ such that $ 2 \leq p  <  +\infty $, for all $u_{0} \in L^p(\Gamma(TM)) \cap L^2(\Gamma(TM))$ and some $C>0$.
\end{theorem}
\begin{proof}
 In the case $p=2$, the above estimate is a direct consequence of the energy estimate for the Stokes equations \eqref{NSprew1}. 
 Indeed, multiplying $u$ in \eqref{NSprew1} and then integrating on $M$  by part combining with the Bochner identity \eqref{bochner}, we have 
 the following energy estimate
 \begin{equation}\label{ENB}
 \|u(t)\|_{L^2}^2 + \int_{0}^t\big( \| \nabla u (s) \|_{L^2}^2  +  c_0 \,  \|u(s) \|_{L^2}^2 \big)\, ds \leq  \|u_{0}\|_{L^2}^2.
\end{equation}
 
  For $p>2$, the $L^p \to L^p$ type estimate that we used previously does not yield a good result for large times due to the additional
   term $Bu$ that does not vanish.  
  In the following, we shall use an argument that  relies on the $L^2 \rightarrow L^p$ dispersive estimate. This is the reason for which we also need the initial data to be in $L^2$.
  By  dispersive estimates \eqref{dispersiveestimatesBSmall}, we have
     \begin{equation}
 \|u(t)\|_{L^p} \leq c_n(t)^{\left(\frac{1}{2}-\frac{1}{p}\right)} e^{- \beta t}   \|u_{0}\|_{L^2}, \quad \forall 0<t \leq 2, \,\,  p\geq2.
\end{equation}
     We use the semi-group property combined with the last estimate  at $t=1$ and the $L^2 \to L^2$ estimates \eqref{BLpestimates} for $t>1$, then
$$ e^{t (\overrightarrow{\Delta}+r -B) } = \, e^{(\overrightarrow{\Delta}+r -B) } \, (e^{{(t-1)} (\overrightarrow{\Delta} +r-B)} ): L^2 \rightarrow L^2 \rightarrow L^p$$
is bounded and  we have 
 \begin{equation}
 \|u(t)\|_{L^p} \leq C e^{- \beta t}   \|u_{0}\|_{L^2}, \quad \forall t\geq 2 \,\, \mathrm{and} \,\, p\geq2.
\end{equation}
Combining the  last estimate and \eqref{dispersiveestimatesBSmall} with $p=q$, we deduce
\begin{equation}
\| u(t)  \|_{L^p} \leq \, C e^{-  \beta t } \, \left(\| u_{0} \|_{L^p}+ \| u_{0} \|_{L^2}\right),
 \quad \forall t> 0 \,\, \mathrm{and} \,\, p\geq2.
\end{equation}
 \end{proof}
 
   \begin{corollary}
   Assuming {\bf (H1-4)},  there exist $\beta \geq c_0 >0$ such that the solution of  the Cauchy problem \eqref{NSprew1} satisfies the following  dispersive estimates 
\begin{equation}\label{BLpqestimates}
\| u(t)  \|_{L^q} \leq \, C   c_n(t)^{ \left({1\over p} - { 1 \over q } \right )}  e^{-  \beta t } \, \left(\| u_{0} \|_{L^p}+ \| u_{0} \|_{L^2}\right), \quad \forall \,\,t>0,
\end{equation}
for every   $p,q$ such that $ 2 \leq p \leq q <  +\infty, $ 
 for all $u_{0} \in L^p(\Gamma(TM)) \cap L^2(\Gamma(TM))$ and some $C>0 $ and
\begin{equation}\label{BLpqestimates2}
\| u(t)  \|_{L^q} \leq \, C   c_n(t)^{ \left({1\over p} - { 1 \over q } \right )}  e^{-  \beta t } \, \| u_{0} \|_{L^p} , \quad \forall \,\,t>0,
\end{equation} 
for every   $p,q$ such that $ 1< p \leq 2 \leq q <  +\infty, $ 
 for all $u_{0} \in L^p(\Gamma(TM)) $ and some $C>0, $
 with  $ c_n(t)=  \max \big({1 \over  t^{n\over 2}}, 1 \big)$.
     \end{corollary}
  \begin{proof}
By using the semi-group property combined with  dispersive estimates \eqref{dispersiveestimatesBSmall}  at $t=1$ and the $L^2 \cap L^p \to L^p$ estimates \eqref{BLpestimates} for $t>1$
$$ e^{t (\overrightarrow{\Delta}+r -B) } = \, e^{(\overrightarrow{\Delta}+r -B) } \, (e^{{(t-1)} (\overrightarrow{\Delta} +r-B)} ): L^2 \cap L^p \rightarrow  L^p \rightarrow L^q$$
is bounded and we have that   
 \begin{equation}
 \|u(t)\|_{L^q} \leq C e^{- \beta t}   \left(\| u_{0} \|_{L^p}+ \| u_{0} \|_{L^2}\right)
  \end{equation}
  for $t\geq 2$ and for every   $p,q$ such that $ 2 \leq p \leq q <  +\infty.$
Combining the  last estimate and \eqref{dispersiveestimatesBSmall} in Proposition \ref{propB1}, we deduce \eqref{BLpqestimates}.

To prove \eqref{BLpqestimates2}, we note that   \eqref{BLpqestimates}  yields a $L^2 \rightarrow L^p$ estimate valid for all positive
times and for $p\geq 2$. By duality, we deduce   the $ L^{p'} \rightarrow L^2$ estimate and we finally get \eqref{BLpqestimates2}
by using the semigroup property $ e^{t (\overrightarrow{\Delta}+r -B) } =  e^{{t\over2} (\overrightarrow{\Delta}+r -B) } 
e^{{t\over2} (\overrightarrow{\Delta}+r -B) }$.

 \end{proof}

\begin{prop}
 \label{propB2}
Assuming {\bf (H1-4)},  the solution of \eqref{NSprew1} satisfies the following smoothing estimates 
\begin{equation}\label{Bsmoothingestimates}
 \| \nabla u (t) \|_{L^p } \leq { C { \frac{1}{\sqrt{t}}}}  \| u_0 \|_{L^p} \quad \forall \,\, 0< t \leq 2,
 \end{equation}
for every $ 1 < p \leq q   < + \infty$ and for all $u_{0} \in L^p(\Gamma(TM)).$
\end{prop}
\begin{proof}
By using the Duhamel formula we can write 
$$ \nabla u(t)=  \nabla e^{t ( \overrightarrow{\Delta} + r)} u_0 + \int_0 ^t   \nabla  e^{(t -\tau) ( \overrightarrow{\Delta} + r)} Bu(\tau) \, d\tau.$$
Thanks to  \eqref{smoothingestimates} in Theorem \ref{theolin2}, we have 
\begin{equation}\label{LqestB2}
\|  \nabla u(t)\|_{L^p} \leq C  { \frac{1}{\sqrt{t}}}  \|u_0\|_{L^p} +  \int_0 ^{t} C { \frac{1}{\sqrt{t-\tau}}}  \|u(\tau) \|_{L^p} d \tau 
\end{equation}
and by using \eqref{dispersiveestimatesBSmall} we find
$$\| \nabla u(t)\|_{L^p} \leq C  { \frac{1}{\sqrt{t}}} \|u_0\|_{L^p}, \quad \forall 0<t \leq 2.$$
\end{proof}

 \begin{theorem}
Assuming {\bf (H1-4)}, there exist $\beta \geq c_0 >0$ such that the solution of  the Cauchy problem \eqref{NSprew1} satisfies the following  estimates 
\begin{equation}\label{BLSpestimates}
\| \nabla u(t)  \|_{L^q} \leq \, C \,  c_n(t)^{ \left({1\over p} - { 1 \over q } + {1 \over n} \right )} \,  e^{-  \beta t } \, \left(\| u_{0} \|_{L^p}+ \| u_{0} \|_{L^2}\right), \quad \forall \,\,t>0,
\end{equation}
 for every   $p,q$  such that $  2 \leq p \leq q <  +\infty, $   for all $u_{0} \in L^p(\Gamma(TM)) \cap L^2(\Gamma(TM))$ and some $C>0$ and
\begin{equation}\label{BLSpqestimates2}
\| \nabla u(t)  \|_{L^q} \leq \, C   c_n(t)^{ \left({1\over p} - { 1 \over q } + {1 \over n} \right )}  e^{-  \beta t } \, \| u_{0} \|_{L^p}, \quad \forall \,\,t>0,
\end{equation} 
for every   $p,q$ such that $  1< p \leq 2 \leq q <  +\infty, $ 
 for all $u_{0} \in L^p(\Gamma(TM))$ and some $C>0, $
 with  $ c_n(t)=  \max \big({1 \over  t^{n\over 2}}, 1 \big)$.
 \end{theorem}
\begin{proof}
  We use the semi-group property combined with  smoothing estimates \eqref{Bsmoothingestimates}  at $t=1$ and previous estimates \eqref{BLpqestimates} for $t>1$, then
$$  e^{t (\overrightarrow{\Delta}+r -B) } = \,  e^{(\overrightarrow{\Delta}+r -B) } \, (e^{{(t-1)} (\overrightarrow{\Delta} +r-B)} ): L^2 \cap L^p \rightarrow L^q \rightarrow W^{1,q}$$
is bounded and  we have 
 \begin{equation}\label{Bsmo}
 \| \nabla u(t)\|_{L^q} \leq C    c_n(t-1)^{ \left({1\over p} - { 1 \over q } \right )}  e^{- \beta (t-1)}   \left(\| u_{0} \|_{L^p}+ \| u_{0} \|_{L^2}\right), \quad \forall t\geq 2 \,\, \mathrm{and} \,\, 2 \leq p \leq q <  +\infty.
\end{equation}
 By using again the semi-group property combined with  smoothing estimates \eqref{Bsmoothingestimates}  and dispersive estimates \eqref{dispersiveestimatesBSmall} for short time $0<t\leq 2$, we have that
$$  e^{t (\overrightarrow{\Delta}+r -B) } = \,  e^{ \frac{t}{2}(\overrightarrow{\Delta}+r -B) } \, (e^{{\frac{t}{2}} (\overrightarrow{\Delta} +r-B)} ):   L^p \rightarrow L^q \rightarrow W^{1,q}$$
is bounded and  we obtain
 \begin{equation*}
 \| \nabla u(t)\|_{L^q} \leq C    c_n(t)^{ \left({1\over p} - { 1 \over q } + {1 \over n} \right )}   \| u_{0} \|_{L^p}, \quad \forall \, 0<t\leq 2 \,\, \mathrm{and} \,\,  1 < p \leq q <  +\infty.
\end{equation*}
Combining the  last estimate and \eqref{Bsmo}, we can conclude the proof of \eqref{BLSpestimates}.

To prove \eqref{BLSpqestimates2}, we use the same splitting, the only difference is that for the large time estimate we use \eqref{BLpqestimates2} in place of \eqref{BLpqestimates}. 
\end{proof}

\section{Fujita-Kato theorems on manifolds}\label{sectionKato}

\subsection{ Strong solutions for Navier-Stokes on  Einstein manifolds with negative curvature}\label{katofixed}

We will first  restrict  our attention to the case of non-compact Riemannian manifolds $M$ for which the Ricci tensor $\Ric$ is a negative constant scalar multiple $r$ of the metric. 
By using \eqref{NLdiv}, in this case is more convenient to rewrite  the  nonlinear Cauchy problem for $u(t, \cdot) \in \Gamma( T M)$ in the following way  
\begin{equation}\label{NSnon}
 \left\{ \begin{array}{ll} \partial_{t} u  -\left(\overrightarrow{\Delta} u  +r u \right)= -    \mathbb{P}(\dive \, ( u \otimes u) ), \quad \mathbb{P}v=  v +  \grad (-\Delta_{g})^{-1} \dive v, \\
   u_{|t=0}= u_{0}, \, \dive u_{0}= 0.
   \end{array} \right.
\end{equation}
We recall the definition of well-posedness:
\begin{definition}
 the  Cauchy problem is  locally well-posed on a Banach space $X$  if for any bounded subset $B$ of $X$, there exists  $T>0$ and 
 a Banach space $X_{T}$ continuously contained into $ \mathcal{C}([0, T], X)$ such that:
 \begin{itemize}
 \item[i)] for any Cauchy data  $u_{0}(x) \in B$, {\eqref{NSnon}} has a unique solution $u(t,x) \in X_{T}$;
 \item[ii)] the flow map $ u_{0} \in B \rightarrow u(t,x) \in X_{T}$ is continuous. 
 \end{itemize}  
 We say that the problem is  globally well-posed if these properties hold for $T=+\infty$.
  \end{definition}
  
\begin{theorem} ({\bf{ Well-posedness on $M$}}) \label{Theorconstant}
For every {$u_{0}\in L^n (\Gamma(T M))$}, with $\dive u_{0}=0$, there exists $T>0$ and a {unique solution $u$}  of
 the incompressible Navier-Stokes equations  such that {$u\in \mathcal{C}([0, T], L^n(\Gamma(T M)))\cap X_T$}.\\
Moreover, there exists $\delta>0$ such that if {$\|u_{0}\|_{L^n(\Gamma(T M))} \leq \delta $} then the above solutions are {global in time}.\\
In { dimension $2$}, the solutions are {global for large data}.
\end{theorem}
The Banach space $X_{T}$ will be defined below.
\begin{proof}
We have to solve the  fixed point problem
$$ u(t) =  e^{t  (\overrightarrow{\Delta} +r)} u_{0} - \int_{0}^t   e^{(t-\tau)  (\overrightarrow{\Delta}+r) }   \mathbb{P}(\dive ( u \otimes u) )(\tau)\, d\tau
 = u_{1} + B(u,u)(t).
$$ 
We use the following classical variant of the  Banach fixed point Theorem :
\begin{lemma}\label{lemmafixed}
 Consider $X$  a Banach space and $B$ a bilinear operator such that
 {$$ \forall u, \, v \in X, \quad \|B(u,v)\|_{X} \leq \gamma \|u\|_{X}\, \|v\|_{X}, $$}
 then, for every $u_{1}\in X$, such that  {$ 4 \gamma \|u_{1}\|_{X}<1$,} the sequence defined by 
 $$  u_{n+1} =u_{1}+ B(u_{n}, u_{n}), \quad u_{0}=0$$
 converges to {the unique solution} of 
 {$$ u=  u_{1} + B(u,u)$$}
  such that {$2 \gamma  \|u\|_{X} <1$.}
\end{lemma}
We notice that this is the Kato's scheme, which consist in finding a  family of spaces $(X_T)_{T>0}$ such that the bilinear operator $B$ maps $X_T \times X_T$ into $X_T$ continuously. This will produce automatically local or global well-posedeness result.
 Then to prove the {continuity of $B(u,u)$ on $X_{T}$}, we use this Lemma  with 
{$$ X_{T}= \left\{ u \in L^\infty_{loc}([0, T], L^q(\Gamma(T M)) | \,   c_n(t)^{-({1 \over n} - {1 \over q } ) }  e^{\beta t}\|u(t)\|_{L^q} \in L^\infty(0,T) \right\}$$}
for some  {$n<q<+\infty$} and $\beta$ adapted to the large time decay rate of our dispersive estimates.
We recall that $ c_n(t)=  C_{n} \max{\left({1 \over   t^{n\over 2}}, 1 \right)}$. We notice that the $\overrightarrow{\Delta} +r$ has the property of commuting with the projection $\mathbb{P}$  as long as  $M$ has no boundary. Using this fact, we  write thanks to the $L^p$ boundedness of the Riesz transform (see \cite{Lohoue})
and our smoothing estimates \eqref{smononl}
\begin{align*}
 \|B(u,u)(t)\|_{L^q}  & \leq C \int_{0}^t { c_n(t-\tau)^{ {1 \over n} + {1 \over q}  } }e^{- \beta (t-\upsilon)} \| u \otimes u (\tau)\|_{L^{q \over 2 }} \, ds  \\ 
& \leq  C\int_{0}^t {  c_n(t-\tau)^{ {1 \over n} + {1 \over q}  } }e^{- \beta (t-\tau)}  \left( {e^{- \beta \tau}  c_n(\tau)^{  {1 \over n} - {1 \over q}}}\right)^2 \, d\tau \| u \|_{X_{T}}^2.
\end{align*}
Consequently, we obtain for any {$q>n$}
$$  \|B(u,u)\|_{X_{T}} \leq C  \| u \|_{X_{T}}^2. $$
By using the Lemma \ref{lemmafixed}, we get  a { solution in $X_{T}$ \, if  \, $ 4 C \|u_{1}\|_{X_{T}} <1$.}
If {$u_{0} $ is small} in ${L^n(\Gamma(T M)} $, this is {true with $T= +\infty$}. 
  If {$u_{0}$ is not small}, we use as usual that
 {$ \lim_{T \rightarrow 0} \|u_{1}\|_{X_{T}} = 0$} to get the local (in time) well-posedness result in $X_T$.
By classical argument we finally get  $u \in \mathcal{C}([0, T], L^n).$

 In dimension $n=2$,  we can prove that the above solutions are global (in time) with any initial data  $u_0 \in L^2.$
 Though the energy estimate gives an unconditional control of the $L^2$ norm, this is not sufficient to obtain global existence since
 in the above fixed point argument the existence time $T$ does not depend only  on the $L^2$ norm of $u_0$. This is due to the fact that
  $L^2$ is the  critical space in dimension two. Nevertheless, we can overcome this problem with the following classical argument.
   We first note that at time $T$ the above solution is such that $u(T) \in H^1$ due to the smoothing effect.  By an easy fixed point Theorem, 
    we can then continue this solution in $H^1$  on $[T, T_1]$ with an existence time that only depends on the norm of the initial data in $H^1$.
     Consequently, we can obtain global existence if we derive an a priori bound on the $H^1$ norm of $u$.  Thanks to the boundedness
      of the Riesz transform that gives the estimate
      $$ \| \nabla u \|_{L^2} \lesssim \| du^\flat \|_{L^2} + \|u\|_{L^2}$$
      it actually suffices to get an estimate on the $L^2$ norm of $du^\flat$. To get this a priori estimate, we can observe that in terms
       of differential forms, the Navier-Stokes equation can be written as
       $$ \partial_t u^\flat + L_u u^\flat  + {1 \over 2}  d |v|^2 + dp = \delta_H u^\flat + 2 r u^\flat$$
       where $L_u$ is the Lie derivative. This yields for $\eta= d u^\flat$ the equation
       $$ \partial_t \eta + L_u \eta = \Delta_H \eta + 2 r \eta$$
        and hence by identifying $\eta$ with a scalar function $\omega$, we obtain
        $$ \partial_t \omega +  d \omega (v) = \Delta_g \omega + 2 r \omega.$$
         Since $v$ is divergence free, we deduce from this equation  that
         $$ \| \omega (t) \|_{L^2} \leq  \|\omega (s)\|_{L^2} ,\quad t  \geq s
$$
and the result follows. 
  
\end{proof}

\subsection{ Strong solutions for Navier-Stokes on  more general non-compact manifolds}\label{katofixed2}

 In this section, we shall study the well-posedness on suitable Banach spaces of the following non-linear Cauchy problem on more general non-compact Riemannian manifolds $M$ satisfying our assumptions {\bf(H1-4)} :
 \begin{equation}\label{NSprew1bis}
 \left\{ \begin{array}{ll}   \partial_t u  -  \overrightarrow{\Delta} u - r(u)  - B u=  - \mathbb{P}\left[ \nabla_u u \right] ,  \quad \mathbb{P}v=  v +  \grad (-\Delta_{g})^{-1} \dive v, \\
  \dive {u_0}=   0, \\   u_{|t=0}= u_{0}, \quad u_0 \in \Gamma(TM),
 \end{array} \right.
  \end{equation}
  where  $ Bu = - 2 \grad (-\Delta_g)^{-1}\dive (ru)$ and as remarked before  $B$ and $\mathbb{P}$ are bounded as  linear operators  $L^p \to L^p$ for every $p$, $ 1<p<+\infty$   (see again \cite{Lohoue}).
 We notice that the $(\overrightarrow{\Delta} +r -B) $ has not the property of commuting with the operator $\mathbb{P}$ on $M$, we thus have to modify the  functional space where we use the fixed point argument.
  
\begin{theorem} ({\bf{ Well-posedness on $M$}})
For every {$u_{0}\in L^n (\Gamma(T M)) \cap L^2 (\Gamma(T M)) $}, with $\dive u_{0}=0$, there exists $T>0$ and a {unique solution $u$}  of
 the incompressible Navier-Stokes equations  \eqref{NSprew1bis} such that {$u\in \mathcal{C}([0, T],  L^n (\Gamma(T M)) \cap L^2 (\Gamma(T M))\cap X_T$}.\\
Moreover, there exists $\delta>0$ such that if {$\|u_{0}\|_{L^n(\Gamma(T M))} + \|u_{0}\|_{L^2(\Gamma(T M))}   \leq \delta $} then the above solutions are {global in time}.\\
In { dimension $2$}, the solutions are {global for large data}.
\end{theorem}
\begin{proof}
We have to solve the  fixed point problem
$$ u(t) =  e^{t  (\overrightarrow{\Delta} +r - B)} u_{0} - \int_{0}^t   e^{(t- \tau)  (\overrightarrow{\Delta}+r - B)  }   \mathbb{P}( \nabla_u u   )(\tau)\, d\tau
 = u_{1} + B(u,u)(t).
$$ 
 We use again the previous Lemma \ref{lemmafixed} with the following functional space $X_T$:
\begin{align*}
&X_{T}=  \left\{ u \in L^\infty_{loc}([0, T], L^q(\Gamma(T M)), \,
 \nabla u \in L^\infty_{loc}([0, T], L^{\tilde{q}}(\Gamma(T M)) \cap L^\infty_{loc}([0, T], L^s(\Gamma(T M))  | \right.\\
   &\left. e^{\beta t}  c_n(t)^{- ({1 \over n} - {1 \over q } ) }  \|u(t)\|_{L^q} +   e^{\beta t} c_n(t)^{- ({2 \over n} - {1 \over {\tilde{q}}} ) }  \| \nabla u(t)\|_{L^{\tilde{q}}} +   e^{\beta t}  c_n(t)^{- ({2 \over n} - {1 \over s })}  \| \nabla u(t)\|_{L^s}   \in L^\infty(0,T) \right\}
\end{align*}
for some suitable $ n<q,  \widetilde{q}, s<+\infty $ and $\beta$ adapted to the large time decay rate of our dispersive estimates.
We recall that $ c_n(t)=  C_{n} \max{\left({1 \over   t^{n\over 2}}, 1 \right)}$. 
  Thanks to the $L^p$ boundedness of the Riesz transform (see \cite{Lohoue})
and our dispersive estimates \eqref{BLpqestimates}, we have
\begin{align}
\label{B1} \|B(u,u)(t)\|_{L^q}  & \leq C \int_{0}^t { c_n(t- \tau)^{ {1 \over r} - {1 \over q}  } }e^{- \beta (t- \tau)} \left( \| \nabla_u u (\tau)\|_{L^{r }} +  \| \nabla_u u (\tau)\|_{L^{2 }} \right) \, d\tau 
\end{align}
for $q\geq r \geq 2$ and by using the H\"older inequality and the definition of our $X_T$ norm, we obtain
\begin{align*}
 & \leq C \int_{0}^t { c_n(t-\tau)^{ {1 \over r} - {1 \over q}  } }e^{- \beta (t-\tau)} \left( \| u (\tau)\|_{L^{q}} \| \nabla u (\tau)\|_{L^{s }} + \|  u (\tau)\|_{L^{q }}  \|  \nabla u (\tau)\|_{L^{\widetilde{q} }} \right) \, d\tau  \\ 
  & \leq C \int_{0}^t  c_n(t-\tau)^{ {1 \over r} - {1 \over q}  }  e^{- \beta (t-\tau)}  e^{- 2 \beta \tau} 
 c_n(\tau)^{ {1 \over n} - {1 \over q}  } \left( c_n(\tau)^{ {2 \over n} - {1 \over s}  } + c_n(\tau)^{ {2 \over n} - {1 \over \widetilde{q}}  } \right)   \, d \tau   \|  u (\tau)\|_{X_T}^2
\end{align*}
with 
\begin{equation}\label{CHolder}
\frac{1}{r}= \frac{1}{q}+ \frac{1}{s}\,, \quad \frac{1}{2}= \frac{1}{q}+ \frac{1}{\widetilde{q}} \,.
\end{equation}
Consequently, this yields
$$ \sup_{t>0} \left[ c_n(t)^{-(\frac{1}{n} - \frac{1}{q})} e^{ \beta t}  \|B(u,u)(t)\|_{L^q} \right]  \leq C  \| u \|_{X_{T}}^2, $$
since 
$$ \sup_{t>0} \, \left[c_n(t)^{-(\frac{1}{n} - \frac{1}{q})} \int_{0}^t  c_n(t-\tau)^{ {1 \over r} - {1 \over q}  }   e^{- \beta \tau} 
 c_n(\tau)^{ {1 \over n} - {1 \over q}  } \left( c_n(\tau)^{ {2 \over n} - {1 \over s}  } + c_n(\tau)^{ {2 \over n} - {1 \over \widetilde{q}}  } \right)   \, d \tau\right] < + \infty$$
 if 
 \begin{equation}\label{Co2}
 q \geq n, \quad q\geq r,\quad 2\leq r < n, \quad s > \frac{n}{2},\quad \frac{1}{r} \leq \frac{1}{s} + \frac{1}{n}.
 \end{equation}
In similar way by using  the smoothing estimates \eqref{BLSpestimates}, we get
\begin{align}
\label{nablaB1}
 \|\nabla B(u,u)(t)\|_{L^s}  & \leq C \int_{0}^t { c_n(t- \tau)^{ {1 \over r} - {1 \over s} + {1 \over n}  } }e^{- \beta (t- \tau)} \left( \| \nabla_u u (\tau)\|_{L^{r }} +  \| \nabla_u u (\tau)\|_{L^{2 }} \right) \, d\tau 
\end{align}
for $s\geq r\geq 2$ and again by using the H\"older inequality and the definition of our $X_T$ norm, we obtain
\begin{align*}
 & \leq C \int_{0}^t { c_n(t-\tau)^{ {1 \over r} - {1 \over s} + {1 \over n}   } }e^{- \beta (t-\tau)} \left( \| u (\tau)\|_{L^{q}} \| \nabla u (\tau)\|_{L^{s }} + \|  u (\tau)\|_{L^{q }}  \|  \nabla u (\tau)\|_{L^{\widetilde{q} }} \right) \, d\tau  \\ 
  & \leq C \int_{0}^t  c_n(t-\tau)^{ {1 \over r} - {1 \over s} + {1 \over n}   }  e^{- \beta (t-\tau)}  e^{- 2 \beta \tau} 
 c_n(\tau)^{ {1 \over n} - {1 \over q}  } \left( c_n(\tau)^{ {2 \over n} - {1 \over s}  } + c_n(\tau)^{ {2 \over n} - {1 \over \widetilde{q}}  } \right)   \, d \tau   \|  u (\tau)\|_{X_T}^2.
\end{align*}
Then 
$$ \sup_{t>0} \left[ c_n(t)^{-(\frac{2}{n} - \frac{1}{s})} e^{ \beta t}  \| \nabla B(u,u)(t)\|_{L^s} \right]  \leq C  \| u \|_{X_{T}}^2, $$
since 
$$ \sup_{t>0} \, \left[c_n(t)^{-(\frac{2}{n} - \frac{1}{s})} \int_{0}^t  c_n(t-\tau)^{ {1 \over r} - {1 \over s} + {1 \over n}  }   e^{- \beta \tau} 
 c_n(\tau)^{ {1 \over n} - {1 \over q}  } \left( c_n(\tau)^{ {2 \over n} - {1 \over s}  } + c_n(\tau)^{ {2 \over n} - {1 \over \widetilde{q}}  } \right)   \, d \tau\right] < + \infty$$
 if 
  \begin{equation}\label{Co3}
 s \geq n, \quad s\geq r,\quad 2\leq r < n, \quad \frac{1}{r} \leq \frac{1}{s} + \frac{1}{n}.
  \end{equation}
Finally, in the same way by using again  the smoothing estimates \eqref{BLSpestimates}, we have
\begin{align*}
 \|\nabla B(u,u)(t)\|_{L^{\widetilde{q}}}  & \leq C \int_{0}^t { c_n(t- \tau)^{ {1 \over r} - \frac{1}{\widetilde{q}}  + {1 \over n}  } }e^{- \beta (t- \tau)} \left( \| \nabla_u u (\tau)\|_{L^{r }} +  \| \nabla_u u (\tau)\|_{L^{2 }} \right) \, d\tau 
\end{align*}
for $ \widetilde{q}\geq r\geq 2$ and again by using the H\"older inequality and the definition of our $X_T$ norm, we obtain
\begin{align*}
 & \leq C \int_{0}^t { c_n(t-\tau)^{ {1 \over r} - \frac{1}{\widetilde{q}} + {1 \over n}   } }e^{- \beta (t-\tau)} \left( \| u (\tau)\|_{L^{q}} \| \nabla u (\tau)\|_{L^{s }} + \|  u (\tau)\|_{L^{q }}  \|  \nabla u (\tau)\|_{L^{\widetilde{q} }} \right) \, d\tau  \\ 
  & \leq C \int_{0}^t  c_n(t-\tau)^{ {1 \over r} - \frac{1}{\widetilde{q}} + {1 \over n}   }  e^{- \beta (t-\tau)}  e^{- 2 \beta \tau} 
 c_n(\tau)^{ {1 \over n} - {1 \over q}  } \left( c_n(\tau)^{ {2 \over n} - {1 \over s}  } + c_n(\tau)^{ {2 \over n} - {1 \over \widetilde{q}}  } \right)   \, d \tau   \|  u (\tau)\|_{X_T}^2.
\end{align*}
Thus
$$ \sup_{t>0} \left[ c_n(t)^{-(\frac{2}{n} - \frac{1}{\widetilde{q}})} e^{ \beta t}  \| \nabla B(u,u)(t)\|_{L^{\widetilde{q}}} \right]  \leq C  \| u \|_{X_{T}}^2, $$
since 
$$ \sup_{t>0} \, \left[c_n(t)^{-(\frac{2}{n} - \frac{1}{\widetilde{q}})} \int_{0}^t  c_n(t-\tau)^{ {1 \over r} - \frac{1}{\widetilde{q}} + {1 \over n}  }   e^{- \beta \tau} 
 c_n(\tau)^{ {1 \over n} - {1 \over q}  } \left( c_n(\tau)^{ {2 \over n} - \frac{1}{s} } + c_n(\tau)^{ {2 \over n} - {1 \over \widetilde{q}}  } \right)   \, d \tau\right] < + \infty$$
 if 
  \begin{equation}\label{Co4}
\widetilde{q} \geq n, \quad \widetilde{q} \geq r,\quad 2\leq r < n, \quad \frac{1}{r} \leq \frac{1}{\widetilde{q}} + \frac{1}{n}.
 \end{equation}
Consequently, by choosing $q,\widetilde{q}, s$ such that the conditions \eqref{CHolder}, \eqref{Co2}, \eqref{Co3}, \eqref{Co4} are verified, we obtain 
$$  \|B(u,u)\|_{X_{T}} \leq C  \| u \|_{X_{T}}^2. $$
By using the Lemma \ref{lemmafixed}, we get  a { solution in $X_{T}$ \, if  \, $ 4 C \|u_{1}\|_{X_{T}} <1$.}
If {$u_{0} $ is small} in $ L^n(\Gamma(T M)) \cap  L^2(\Gamma(T M)) $, this is {true with $T= +\infty$}. 
  If {$u_{0}$ is not small}, we use as usual that
 {$ \lim_{T \rightarrow 0} \|u_{1}\|_{X_{T}} = 0$} to get the local (in time) well-posedness result in $X_T$.
By classical arguments,  we finally get  $u \in \mathcal{C}([0, T], L^n\cap  L^2).$

To handle the two-dimensional case, we can proceed in the same way with 
$X_T$ defined as
\begin{align*}
&X_{T}=  \left\{ u \in L^\infty_{loc}([0, T], L^q(\Gamma(T M)), \,
 \nabla u \in L^\infty_{loc}([0, T], L^{\tilde{q}}(\Gamma(T M)) \cap L^\infty_{loc}([0, T], L^s(\Gamma(T M))  | \right.\\
   &\left. e^{\beta t}  c_n(t)^{- ({1 \over n} - {1 \over q } ) }  \|u(t)\|_{L^q} +      e^{\beta t}  c_n(t)^{- ({2 \over n} - {1 \over s })}  \| \nabla u(t)\|_{L^s}   \in L^\infty(0,T) \right\}
\end{align*}
with $q>2$ and $s>2$. To estimate $B(u,u)$ in $X_T$ the only differences with the previous computations is 
 that when
using the dispersive and smoothing estimates to get \eqref{B1} and  \eqref{nablaB1}, we take $r<2$ and thus we apply 
\eqref{BLpqestimates2}, \eqref{BLSpqestimates2} in place of \eqref{BLpqestimates}, \eqref{BLSpestimates}.

To get the global well-posedness, we use the same argument as in the end of the proof of Theorem \ref{Theorconstant}.

\end{proof}

\section{Remarks on the uniqueness of weak solutions}\label{sectionLeray}

 As shown in \cite{Czubak}, \cite{Khesin},  for  two-dimensional manifolds, one has to be careful with the definition of  Leray type  solutions in order to get uniqueness.
 Indeed, it was proven that for the hyperbolic space  $\mathbb{H}^2$, due to the presence of non-trivial bounded harmonic forms there exists infinitely many
 weak solutions $u \in L^\infty_{T} L^2 \cap  L^2_{T} H^1$ that satisfy the energy inequality:
 \begin{equation}
 \label{energie}
  \|u(t)\|_{L^2}^2 + \int_{0}^t  \| \nabla u(s) \|_{L^2}^2 \, ds \leq \|u(0)\|_{L^2}^2
 \end{equation}
for almost every $t \geq 0$.
A way to recover the uniqueness, by carefully selecting the weak solution was recently proposed in \cite{CzubakChan}. We shall propose another way to recover the uniqueness
 in terms of the regularity of the pressure for two-dimensional manifolds that satisfy {\bf(H1-4)}.
 Let us recall that the pressure is  the solution of the elliptic equation
 \begin{equation}
 \label{eqpression} \Delta_{g} p= -\dive ( \dive (u\otimes u)  - 2 ru).\end{equation}
  If $u$ has the regularity of a Leray solution, $u \in L^\infty_{T} L^2 \cap  L^2_{T} H^1$, then
  $\dive (ru) \in L^2_{T} L^2$ and $u \otimes u  \in L^2_{T} L^2$. Indeed, we have that
  $$ \| u \otimes u  \|_{L^2} \leq \| | u |  \|_{L^4}^2 \leq  C\big( \| \nabla u \|_{L^2} \|u\|_{L^2} + \| u \|_{L^2}^2\big)$$
  thanks to the Gagliardo-Nirenberg inequality in Remark \ref{REMP} (3). This yields
  $$  \| u \otimes u  \|_{L^2_{T} L^2} \leq  C \big( \|u\|_{L^\infty_{T} L^2} \| \nabla u \|_{L^2_{T} L^2} + T^2 \|u \|_{L^\infty_{T} L^2}^2 \big).$$
   Consequently  $ \dive (\dive (u \otimes u)) \in L^2_{T}H^{-2}.$
  Since $\Delta_{g}:\, L^2 \rightarrow H^{-2}$ is an isomorphism by using \cite{Lohoue}, there exists a unique solution $p$ of 
  \eqref{eqpression} such that $ p \in L^{2}_T L^2$.
  This motivates the following definition of Leray weak solutions:
  \begin{definition}
   For every divergence free  $u_{0}\in L^2 \Gamma(TM)$, we shall say that $u \in L^\infty_{T}L^2 \cap L^2_{T}H^1$ is a Leray weak solution of the Navier-Stokes equation
    with initial data $u_{0}$
    if for every $\phi \in \mathcal{C}^1_{c}(\overline{\mathbb{R}}_{+} \times M, TM),$ we have
   \begin{multline}
   \label{defweak}\int_{\mathbb{R_{+}} \times M} \Big(  g( u, \partial_{t} \phi) + g( u\otimes u, \nabla \phi) + p \dive \phi - g( \nabla u , \nabla \phi) +  g(ru, u) \Big) dV_{g} dt
    \\+ \int_{M} g(u_{0}, \phi(0, \cdot) ) dV_{g} = 0
    \end{multline}
    with $p \in L^2_{T}L^2$ the unique solution of the elliptic equation \eqref{eqpression}.
     \end{definition}
    We claim that
  \begin{theorem}
  Assume that $M$ is a two-dimensional complete  simply connected non-compact manifold that satisfy {\bf (H1-4)}. Then, 
    for every divergence free  $u_{0}\in L^2 \Gamma(TM)$, there  exists a unique weak Leray solution.
  \end{theorem}
    \begin{proof}
   There are many classical ways to prove the existence. Note that the strong solutions that we have constructed in section \ref{katofixed} are actually weak solutions, 
   therefore, we shall focus on the uniqueness.
   
   To prove the uniqueness, we shall first prove that our definition of weak solution contains that they satisfy the energy inequality (and even the energy equality).
    We first notice that if $u$ is a weak Leray solution, then $u$ is a solution of
     \begin{equation}
     \label{udistrib}\partial_{t} u = \Delta u + r u - \dive u\otimes u  - \nabla p
     \end{equation}
      in the distribution sense and that $u_{|t=0}= u_{0}$ in the weak sense. Note that the right hand side belongs to $L^2_{T} H^{-1}$ and therefore $\partial_{t}u \in L^2_{T}H^{-1}.$
       We thus obtain that
       \begin{equation}
        \label{ipp}
        \|u(t) \|_{L^2}^2 - \|u(s)\|_{L^2}^2 =  2\int_{s}^t \langle \partial_{t} u , u \rangle \, d\tau
        \end{equation}
        where $\langle \cdot, \cdot \rangle$ is the duality bracket $H^1-H^{-1}$.
         In particular, we obtain that $u \in \mathcal{C}L^2$.
         
    Next, thanks to our assumptions {\bf(H1-4)}, we have that $\mathcal{C}^1_{c}(\overline{\mathbb{R}}_{+} \times M)$ is dense in $ X= \mathcal{C}_{loc}([0, +\infty[, L^2) \cap L^2_{loc}(\mathbb{R}_{+},  H^1) \cap H^1_{loc}(\mathbb{R}_{+},  L^2)$.
    Moreover, all the bilinear terms that appear in the definition \eqref{defweak} are continuous on $X \times X$,  and the trilinear term
     $$ B(u,v,\phi)= \int_{\mathbb{R} \times M}  \Big( g( u\otimes v, \nabla \phi)  - \Delta_{g}^{-1}\big( \dive \dive (u \otimes v) \big) \dive \phi \Big)$$
      is continuous on $X \times X \times X$ as a consequence of the Gagliardo-Nirenberg inequality since
      $$ \|u \otimes v \|_{L^2_{T}L^2} \leq C(T) \big(  \|u \|_{L^\infty_{T} L^2}  \|u\|_{L^2_{T} H^1}\big)^{1 \over 2} \big(\|v \|_{L^\infty_{T} L^2} \|v \|_{L^2_{T} H^1}\big)^{1 \over 2}.$$
       This yields that in our definition of weak solution we can take $\phi \in  X$.
       In addition, since $u \in X$, we can  use the definition for  $\phi=  (\rho_{\varepsilon} *1_{[0, T]}) u$, for every $T >0$ fixed. By using \eqref{ipp}, we obtain by taking $\varepsilon$ to zero that 
       $${ 1 \over 2} \|u(T) \|_{L^2}^2  + \int_{0}^T \int_{M} \big(| \nabla u|^2 - g(ru,u) - g( u\otimes u, \nabla u) - p \, \dive u \big) \, dV_{g}\, dt  ={ 1 \over 2} \|u(0)\|_{L^2}^2.$$
         Since $u$ solves \eqref{udistrib} in the distribution sense, we obtain by taking the divergence that 
         $$ \partial_{t} \dive u - \Delta \dive u = 0.$$
          This proves that $\dive u \in L^2_{T}L^2$ is a solution of the heat equation with  zero initial data. Consequently, $u$ stays divergence free for all times.
           Thus, we obtain 
           $$ \int_{M} g( u \otimes u, \nabla u) = -  \int_{M} g(\nabla_{u} u, u) =  {1 \over 2}\int_{M} |u|^2 \dive u\, dV_{g} = 0.$$
            Consequently, we have proven that 
           $$ { 1 \over 2}  \|u(T) \|_{L^2}^2  + \int_{0}^T \int_{M} \big(| \nabla u|^2 - g(ru,u) \big) \, dV_{g}\, dt  = { 1 \over 2 }\|u(0)\|_{L^2}^2$$
           which is the energy equality.   
           
         Now, consider $u$, $v$ two Leray weak solutions. We can take $\phi=   (\rho_{\varepsilon} *1_{[0, T]}) (u-v)$ and let $\varepsilon$ to zero  to  get that
         \begin{multline*}  (u(T), u(T)- v(T))_{L^2} -  \int_{0}^T\langle u,  \partial_{t}(u-v)\rangle dt  \\+ \int_{0}^T \int_{M}  \Big(g( \nabla u, \nabla (u-v)) - g(ru, u-v))- g(u\otimes u, \nabla (u-v) )\Big) dV_{g} dt = 0
             \end{multline*}
           and that
            \begin{multline*} 
           (v(T), u(T)- v(T))_{L^2} -   \int_{0}^T\langle v,  \partial_{t}(u-v\rangle) dt  \\+ \int_{0}^T \int_{M}  \Big(g( \nabla v, \nabla (u-v)) - g(rv, u-v))- g(v\otimes v, \nabla (u-v) )\Big) dV_{g} dt = 0
          \end{multline*}
        since as already observed $u$ and $v$ are divergence free.
         Next, we can subtract the two identities to obtain
        \begin{multline}
        \label{unique1}  { 1 \over 2} \|u(T) - v(T) \|_{L^2}^2 + \int_{0}^T \big(\| \nabla (u-v) \|_{L^2}^2   + c_{0} \| u-v \|_{L^2}^2\big)\, dt
          \leq  { 1 \over 2} \|u_{0} - v_{0} \|_{L^2}^2  \\+ \int_{0}^T \int_{M} |  g(\nabla_{u}u, u-v) - g(\nabla_{v}v, u-v) | \, dV_{g} dt.
          \end{multline}
           By using that
           $$  g(\nabla_{u}u, (u-v) - g(\nabla_{v}v, u-v) = g( \nabla_{u}(u-v), (u-v) ) + g(\nabla_{u} v - \nabla_{v} v, u-v),$$ we
           obtain 
          \begin{multline}  { 1 \over 2} \|u(T) - v(T) \|_{L^2}^2 + \int_{0}^T \big( \| \nabla (u-v) \|_{L^2}^2  + c_{0} \| u-v \|_{L^2}^2\big) \, dt
          \leq  { 1 \over 2} \|u_{0} - v_{0} \|_{L^2}^2  \\+ \int_{0}^T \int_{M} |  g(\nabla_{u}v - \nabla_{v} v, (u-v) ) | \, dV_{g} dt.
          \end{multline}
           To conclude, we can use again the Gagliardo-Nirenberg inequality  which yields
          \begin{align*}  \int_{0}^T \int_{M} |  g(\nabla_{u}v - \nabla_{v} v, (u-v) ) | \, dV_{g} dt
            &  \leq   \int_{0}^T \int_{M}  | \nabla v| |u-v|^2 \, dV_{g} dt  \\
             & \leq \int_{0}^T  \| \nabla v\|_{L^2} \big( \| \nabla (u-v) \|_{L^2} + \| u-v\|_{L^2} \big) \|u-v\|_{L^2} dt.
             \end{align*}
             By using the Young inequality,  we obtain  from \eqref{unique1} 
            $$  { 1 \over 2} \|u(T) - v(T) \|_{L^2}^2       \leq  { 1 \over 2} \|u_{0} - v_{0} \|_{L^2}^2 +  C \int_{0}^T \| \nabla v \|_{L^2}^2  \|u-v\|_{L^2}^2 dt, \quad \forall T \geq 0$$
             and hence from the Gronwall inequality, we have
             $$ \|u(T)- v(T)\|_{L^2}^2 \leq \|u_{0} - v_{0} \|_{L^2}^2 e^{\int_{0}^T C  \| \nabla v \|_{L^2}^2 dt}, \quad \forall T \geq 0.$$
  In particular, if $u_{0}= v_{0}$, we obtain that $u(T)= v(T)$  for all positive times.

  \end{proof}
  
    As a final remark, we can analyze how  the counterexample given in \cite{Khesin} in the case of $\mathbb{H}^2$  is excluded by our definition of Leray weak solution.
 The velocity field $u$ was chosen under the form
 $$ u= f(t)  (d\Phi)^\sharp$$
  where $\Phi$ is an harmonic function such that $d\Phi \in L^2$ and $f$ is an arbitrary  function of time. In order, to ensure that $v$ is a solution,  the pressure $p$ is chosen as
  $$ p = (2 f(t) - f'(t))\Phi - {1 \over 2} f(t)^2 | d\Phi|^2.$$
  The restriction given by the energy inequality is not sufficient to ensure that the time profile $f$ is completely determined.
   In this construction $d\Phi \in L^2$ but $\Phi$ itself does not belong to $L^2$, consequently, if we require that $p \in L^2_{T}L^2$, then we necessarily have
   $ f' = 2f$
   and thus $f(t) = f(0) e^{2 t}$. This determines completely $u$ from its initial value.
   
   Let us finally note that the definition of weak solutions in \cite{CzubakChan} leads to the same selection of the velocity in the analysis of this counterexample.

\vspace{0,5cm}

{\it Acknowledgements.

Research  supported by the ANR project "Harmonic Analysis at its boundaries".
ANR-12-BS01-0013-01.
The author benefited from two semesters of D\'el\'egation from the CNRS (2013/2014), France.}


\begin{thebibliography}{9999}

\bibitem{Abidi}
 H. Abidi, 
\textit{RŽsultats de rŽgularitŽ de solutions axisymŽtriques pour le systme de Navier-Stokes\/},
Bull. Sci. Math. 132, No. 7 (2008), 592-624. 

\bibitem{AnkerJi}
J.-Ph. Anker, L. Ji, 
\textit{Heat kernel and Green function estimates on noncompact symmetric spaces\/}, 
Geom. Funct. Anal. 9 (1999), 1035-1091.

\bibitem{APV1}
J.-Ph. Anker, V. Pierfelice, M. Vallarino, 
\textit{The Schr\"odinger equation on Damek-Ricci spaces\/}, 
Comm. Part. Diff. Eq.  36, No. 6 (2011), 976-997.


\bibitem{Auscher-Coulhon-Duong-Hoffman}
P. Auscher, T. Coulhon, X.T. Duong, S. Hofmann, 
\textit{Riesz transform on manifolds and heat kernel regularity\/},  
Ann. Sci. ƒcole Norm. Sup. (4) 37, No. 6 (2004), 911-957. 


\bibitem{Auscher}
P. Auscher, D. Frey,
\textit{A new proof for Koch and Tataru's result on the well-posedness of Navier-Stokes equations in $BMO^{-1}$\/}, 
preprint (2013) arXiv:1310.3783. 



\bibitem{Bakry}
D. Bakry, M. Ledoux,
\textit{A logarithmic Sobolev form of the
Li-Yau parabolic inequality\/}, 
Rev. Mat. Iberoamericana 22, No. 2 (2006), 683-702.


\bibitem{Bourgain-Pavlovic}
J. Bourgain, N. Pavlovi\'c,
\textit{Ill-posedness of the Navier-Stokes equations in a critical space in 3D.\/}, 
J. Funct. Anal. 255, No. 9 (2008), 2233-2247.


\bibitem{CKN}
L. Caffarelli, R. Kohn, L. Nirenberg,
\textit{Partial regularity of suitable weak solutions of the Navier-Stokes equations\/}, 
 Comm. Pure Appl. Math. 35, No. 6 (1982), 771-831. 

\bibitem{CannoneMeyer}
M. Cannone, Y. Meyer, 
\textit{Littlewood-Paley decomposition and Navier-Stokes equations\/},  
Methods Appl. Anal. 2, No. 3 (1995), 307-319. 

\bibitem{Carron}
G. Carron,  
\textit{EstimŽes des noyaux de Green et de la chaleur sur les espaces symŽtriques\/},  
 Anal. PDE 3,  No. 2 (2010), 197-205.


\bibitem{CheminGallagher}
J.-Y. Chemin, I. Gallagher, M.  Paicu,
 \textit{Global regularity for some classes of large solutions to the Navier-Stokes equations\/},
  Ann. of Math. (2) 173, No. 2 (2011), 983-1012. 

\bibitem{Coulhon}
T. Coulhon, 
 \textit{Heat kernel estimates, Sobolev-type inequalities and Riesz transform on noncompact Riemannian manifolds\/},
 Analysis and geometry of metric measure spaces, 55-65, CRM Proc. Lecture Notes, 56, Amer. Math. Soc., Providence, RI, (2013).

\bibitem{Czubak}
M. Czubak,  C. H. Chan,
\textit{Non-uniqueness of the Leray-Hopf solutions in the hyperbolic setting\/}, 
 Dynamics of PDE 10, No.1 (2013) 43-77.
 
\bibitem{CzubakChan}
M. Czubak,  C. H. Chan,
\textit{Remarks on the weak formulation of the Navier-Stokes equations on the 2D hyperbolic space\/}, 
preprint (2013) arXiv:1309.3496. 


\bibitem{DR}
E. Damek, F. Ricci,
\textit{A class of nonsymmetric harmonic Riemannian spaces \/}, 
Bull. Amer. Math. Soc.  27 (1992), 139-142.



\bibitem{Druet}
O. Druet,
\textit{Nonlinear analysis on manifolds\/}, 
Lectures Notes.

\bibitem{EbinMarsden}
D.G. Ebin, J.E. Marsden,
 \textit{Groups of diffeomorphisms and the motion of an incompressible fluid\/},
  Ann. of Math. (2) 92 (1970), 102-163. 

\bibitem{Erb}
P. Erbelein,
 \textit{Geometry of non positively curved manifolds\/},
 Chicago Lectures in Mathematics, 449 (1996).
 
 
\bibitem{EusSESve}
L. Escauriaza, G. Seregin, V. \v{S}verak 
\textit{$L^{3, \infty}$ solutions of the Navier-Stokes equations and backward uniqueness\/},
Russian Math. Surveys 58, (2003).

\bibitem{FujitaKato}
H. Fujita, T. Kato, 
\textit{On the Navier-Stokes initial value problem\/},
I, Arch. Rat. Mech. Anal. 16, (1961), 269-315.

\bibitem{FurLemTer} 
G. Furioli, P.G. Lemari\'e-Rieusset and E. Terraneo,
\textit{Uniqueness in $L^3(\mathbb{R}^3)$ and other functional limit spaces for Navier-Stokes equations\/},
Rev. Mat. Iberoamericana 16, No. 3 (2000), 605-667.

\bibitem{GallPlan} 
I. Gallagher, F. Planchon, 
\textit{On Global Infinite Energy Solutions to the Navier-Stokes Equations in Two Dimensions\/},
Arch. Rat. Mech. Anal. 161, (2002), 307-337.

\bibitem{GHL}
S. Gallot, D. Hulin, J. Lafontaine,
\textit{ Riemannian geometry\/}, third edition. Universitext. Springer-Verlag, Berlin (2004).
 
 \bibitem{Giga-Miyakawa}
 Y. Giga, T. Miyakawa, 
\textit{Solutions in $L_r$ of the Navier-Stokes initial value problem\/}, 
 Arch. Rational Mech. Anal. 89, No. 3 (1985), 267-281.

 \bibitem{Grigoryan-Noguchi}
 A. Grigor\'yan, M. Noguchi, 
\textit{The heat kernel on hyperbolic space} 
Bull. London Math. Soc. 30,  No. 6 (1998), 643-650.

\bibitem{Hebey}
E. Hebey,
\textit{Nonlinear Analysis on Manifolds: Sobolev Spaces and Inequalities\/}, 
Courant Lectures in Mathematics, New York (2000) AMS.


\bibitem{H}
S. Helgason,
\textit{Differential geometry, Lie groups, and symmetric spaces\/},
Academic Press (1978) / Amer. Math. Soc. (2001)



\bibitem{hopf}
 E. Hopf,
  \textit{Uber die Aufgangswertaufgabe f\"ur die hydrodynamischen Grundliechungen\/}, 
 Math. Nachr. 4, (1951), 213-231.
 
\bibitem{Jost}
J. Jost.
   \textit{Riemannian geometry and geometric analysis\/}, 
    Universitext. Springer-Verlag, Berlin, fifth edition (2008).


\bibitem{kato}
 T. Kato, 
 \textit{ Strong $L^p$ solutions of the Navier-Stokes equations in Rm with applications to weak solutions\/}, 
  Math. Zeit. 187, (1984), 471Ð480.
  
\bibitem{kenigkoch} 
C.E. Kenig, G. Koch,
 \textit{  An alternative approach to regularity for the NavierStokes equations in critical spaces\/}, 
Ann. I. H. Poincar\'e AN 28, (2011), 159-187.
 
 \bibitem{Khesin}
 B. Khesin, G. Misiolek,
  \textit{ Euler and NavierÐStokes equations on the hyperbolic plane \/}, 
 Proc. Nat. Acad. Sci. (2012).
 
 \bibitem{KochTataru}
 H. Koch, D. Tataru, 
 \textit{  Well Posedness for the Navier Stokes equations\/}, 
  Adv. Math. 157, (2001), 22-35.
  
   \bibitem{Kodaira}
   K. Kodaira, 
   \textit{ Harmonic fields in Riemannian manifolds (generalized potential theory) \/}, 
    Ann. of Math. 2 No. 50 (1949) 587-665.
  
  

  
  
   \bibitem{Lady}
  O. A. Ladyzhenskaya, 
 \textit{   The Mathematical Theory of Viscous Incompressible Flows\/}, 
  (2nd Edition) Vol. 2 Gordon and Breach, New York (1969).
  
  \bibitem{Ladyzhenskaya}
 O. A. Ladyzhenskaya
 \textit{ Unique solvability in large of a three-dimensional Cauchy problem
for the Navier-Stokes equations in the presence of axial symmetry\/},
  Zapisky Nauchnych Sem. LOMI 7 (1968), 155-177.
  
\bibitem{Leray26}
J. Leray,
\textit{Essai sur le mouvement d'un liquide visqueux emplissant l'espace\/},
Acta Mathematica 63 (1993), 193-248.


 \bibitem{Li-Yau}
P. Li, S.-T. Yau, 
  \textit{ On the parabolic kernel of the Schr¬odinger operator\/},
Acta Math. 156, (1986),153-201.


\bibitem{Lin} 
 F. Lin, 
\textit{ A new proof of the Caffarelli-Kohn-Nirenberg theorem\/}, 
 Comm. Pure. Appl. Math. 51, (1998), 241-257.
 
 \bibitem{LM}
P.-L. Lions, N. Masmoudi,
 \textit{ Uniqueness of mild solutions of the Navier-Stokes system in $L^N$\/},
 Comm. Partial Differential Equations 26, No. 11-12 (2001) 2211-2226. 
 

 

 
 \bibitem{Lohoue} 
 N. Lohou\'e
 \textit{ Estimation des projecteurs de De Rham Hodge de certaines vrai\'et\'es riemanniennes non compactes\/},
 Math. Nachr. 279, 3 (2006), 272-298.
 
 \bibitem{McKean}
 H. P. McKean, 
  \textit{An upper bound to the spectrum of A on a manifold of negative curvature\/},
J. Differential Geometry, 4 (1970), 359-366.
 
\bibitem{MitreaTaylor}
M. Mitrea, M. Taylor, 
\textit{Navier-Stokes equations on Lipschitz domains in Riemannian manifolds\/},
Math. Ann., 321 No. 4 (2001) 955-987.

\bibitem{Nash}
J. Nash, 
\textit{Continuity of solutions of parabolic and elliptic equations\/},
Amer. J. Math., 80 (1958) 931-954. (Reviewer: C. B. Morrey Jr.)



\bibitem{Oseen28}
C. Oseen,
\textit{Sur les formules de Green g\'en\'eralis\'ees qui se pr\'esentent dans l'hydrodynamique et sur quelques unes de leurs applications (premi\`ere partie)\/},
Acta Matematica, 34 (1911), 205-284.

\bibitem{Oseen29}
C. Oseen,
\textit{Sur les formules de Green g\'en\'eralis\'ees qui se pr\'esentent dans l'hydrodynamique et sur quelques unes de leurs applications (seconde partie)\/},
Acta Matematica, 35 (1912), 97-192.

\bibitem{Ouhabaz}
E. M. Ouhabaz, 
\textit{$L^p$ contraction semigroups for vector valued functions\/},
 Positivity 3, No. 1 (1999), 83-93. 


     \bibitem{Pedon}
E. Pedon, 
\textit{Harmonic analysis for differential forms on complex hyperbolic spaces\/},
 J. Geom. Phys. 32, No. 2 (1999), 102-130. 

     \bibitem{Planchon}
  F. Planchon,
  \textit{Global strong solutions in Sobolev or Lebesgue spaces for the incompressible Navier-Stokes equations in $\mathbb{R}^3$\/}, 
   Ann. Inst. H. Poincar\'e Anal. Non. Lin\'eaire, 13, (1996), 319-336.
     
       \bibitem{Priebe}
   Volker Priebe,
      \textit{Solvability of the Navier-Stokes equations on manifolds with boundary\/}, 
       Manuscripta Math. 83 No. 2 (1994), 145-159.
       
    \bibitem{Setti}
     A. G. Setti,
     \textit{A lower bound for the spectrum of the Laplacian in terms of sectional and Ricci curvature\/}, 
     Proceedings of the american mathematical society, 112, No 1 (1991), 277- 282.
       
       
       \bibitem{Strichartz}
       R. Strichartz,
         \textit{Analysis of the Laplacian on the Complete Riemannian manifolds\/}, 
       Journal of Functional Analysis 52 (1983), 48-79.
       
        \bibitem{tao}
T. Tao, 
     \textit{A quantitative formulation of the global regularity problem for the periodic Navier-Stokes equation\/}, Dyn.
Partial Differ. Eq. 4, No. 4 (2007), 293-302.
     
    \bibitem{Taylor}
     M. Taylor,
      \textit{Partial Differential Equations III: Nonlinear Equations,\/}
      Nonlinear equations,volume 117 of Applied Math-
ematical Sciences. Springer New York second edition (2011).
  
   \bibitem{Uno}
   M. Uno, 
    \textit{On Sectional Curvature of Boggino-Damek-Ricci Type Spaces,\/}
  Tokyo J. of Math. 23 No. 2 (2000), 417-427.
  
  
      \bibitem{Varopoulos}
 N. Th. Varopoulos,
  \textit{The heat kernel on Lie groups,\/}
   Rev. Mat. Iberoamericana 12,  No. 1 (1996), 147-186.
  
      \bibitem{Weissler}
 F.B.  Weissler, 
 \textit{ The Navier-Stokes initial value problem in $L^{p}$\/}
  Arch. Rational Mech. Anal. 74, No. 3 (1980), 219-230. 
  
    \bibitem{Zuazua}
    E. Zuazua,
    \textit{Large Time Asymptotics For Heat and Dissipative Wave Equations\/}, 
  Lectures Notes 2003, http: //www.uam.es/enrique.zuazua.

\end{thebibliography}
\end{document}